\documentclass [11pt] {article}
\usepackage{amsfonts}
\usepackage{amssymb}
\usepackage{amsmath}
\usepackage{times}
\usepackage{color}
\usepackage{graphicx}
\usepackage{enumerate}

\newtheorem{lemma}{Lemma}[section]
\newtheorem{corollary}{Corollary}[section]
\newtheorem{theorem}{Theorem}

\newtheorem{proposition}{Proposition}[section]
\newtheorem{definition}{Definition}[section]

\def\dim{{\rm dim}}
\def\supp{{\rm supp}}

\def\<{\langle}
\def\>{\rangle}

\def\to{\rightarrow}

\begin{document}

\title{On the almost eigenvectors of random regular graphs}
\author{{\'Agnes Backhausz ~~ Bal\'azs Szegedy}}

\maketitle

\abstract{Let $d\geq 3$ be fixed and $G$ be a large random $d$-regular graph on $n$ vertices. We show that if $n$ is large enough then the entry distribution of every almost eigenvector $v$ of $G$ (with entry sum 0 and  normalized to have length $\sqrt{n}$) is close to some Gaussian distribution $N(0,\sigma)$ in the weak topology where $0\leq\sigma\leq 1$. Our theorem holds even in the stronger sense when many entries are looked at simultaneously in small random neighborhoods of the graph.  Furthermore, we also get the Gaussianity of the joint distribution of several almost eigenvectors if the corresponding  eigenvalues are close. Our proof uses graph limits and information theory. Our results have consequences for factor of i.i.d.\ processes on the infinite regular tree.

\bigskip

\section{Introduction} 

Let $d\geq 3$ and let $G(n,d)$ denote the random $d$-regular graph on $n$ vertices (see e.g.\ the monograph \cite{bollobas}). Equivalently, we can think of $G(n,d)$ as a random model of symmetric $0-1$ matrices in which the row sums are conditioned to be $d$. It is expected that the spectral properties of $G(n,d)$ are closely related to random matrix theory; however, many questions in the area are still open. It is well known that the spectral measure of $G(n,d)$ converges to the so called Kesten--McKay measure in the weak topology as $n$ goes to infinity. This gives an approximate semicircle law if $d$ is large. A famous result by J. Friedman solves Alon's second eigenvalue conjecture showing that $G(n,d)$ is almost Ramanujan \cite{friedman}. Much less is known about the scaled eigenvalue spacing and about the structure of the eigenvectors. Recent results in the area include  \cite{charles, doron} on the second eigenvalue;  \cite{bhky, bauer,  elliot, pal} on eigenvalue spacing, local semicircle law and functional limit theorems; \cite{nalini, brooks, geisinger} on the delocalization of the eigenvectors.

In the present paper we study approximate eigenvectors or shortly: \emph{almost eigenvectors} of $G(n,d)$ i.e.\ vectors that satisfy the eigenvector equation $(A-\lambda I)v=0$ with some small error. Almost eigenvectors are not necessarily close to proper eigenvectors. They are much more general objects. For example any linear combination of eigenvectors with eigenvalues in the interval $[\lambda-\varepsilon,\lambda+\varepsilon]$ is an almost eigenvector with error depending on $\varepsilon$. In general a vector is an almost eigenvector if and only if its spectral measure is close to a Dirac measure in the weak topology. 

We show that despite of this generality, almost eigenvectors of $G(n,d)$ have a quite rigid structure if $n$ is big. Our main result implies that every almost eigenvector (with entry sum 0 and  normalized to have length $\sqrt{n}$) has an entry distribution close to some Gaussian distribution $N(0,\sigma)$ in the weak topology where $0\leq\sigma\leq 1$.  Note that if $\sigma=0$, then the $\ell^2$-weight of the vector is concentrated on a small fraction of the vertices. Such vectors are called localized.  Our main result holds even in a stronger sense where joint distributions are considered using the local structure of the graph. In some sense our result is best possible since there are examples for both localized and delocalized almost eigenvectors (see chapter \ref{prelim}). Note that proper eigenvectors are conjectured to be delocalized.

The issue of eigenvector Gaussianity goes back to random matrix theory. It is not hard to show that in the GUE (Gaussian Unitary Ensemble) random matrix model every eigenvector has a near Gaussian entry distribution. It is much harder to analyze the random model when the elements of the matrix are chosen from a non Gaussian distribution. Nevertheless Gaussianity of the eigenvectors is proved under various conditions for generalized Wigner matrices \cite{bouryau, taovu} and also for various other models (see e.g.\ \cite{bloemen} and \cite{antti, surveyvu} for recent surveys). Sparser models are harder to analyze \cite{huang}. This paper deals with the sparsest case, where the matrix is the adjacency matrix of a random $d$-regular graph with some fixed $d$. In this case there is a stronger meaning of eigenvector Gaussianity. For example it is natural to ask about the Gaussianity of the joint distribution of the entries at the two endpoints of a randomly chosen edge in the graph.  More generally one can look at the joint distribution of the entries in random balls of radius $r$. Our Gaussianity results for almost eigenvectors are established in this strong sense.  Furthermore, it makes sense to study the joint distribution of the entries of several almost eigenvectors. More precisely, a $k$-tuple of almost eigenvectors can be interpreted as a function from the vertices to $\mathbb R^k$. This function evaluated at a randomly chosen vertex gives a probability distribution on $\mathbb R^k$ that we define as the joint distribution of the entries. It is an interesting question whether such joint distributions are Gaussian. We prove this if $k$ is fixed, $n$ is large and the eigenvalues corresponding to the almost eigenvectors are close to each other. Moreover, we also prove this result when the joint distribution of many eigenvectors is considered in random neighborhoods.

The proof of our main theorem is based on the so-called local-global graph limits \cite{BR, HLSz}. However, to keep the paper self-contained, we use a slightly simplified framework (see section \ref{typpro}) optimized for this particular problem. We relate the properties of random regular graphs to random processes on the infinite $d$-regular tree $T_d$. 
Most of the work is done in this convenient limiting framework. 
An invariant process on $T_d$ is a joint distribution of random variables $\{X_v\}_{v\in T_d}$ labeled by the vertices of $T_d$ such that it is invariant under the automorphism group of $T_d$. A special class of invariant processes, called {\it typical processes}, was introduced in \cite{BSz}. Roughly speaking, an invariant process is typical if it can be obtained as the Benjamini--Schramm limit of colored random regular graphs. There is a correspondence between the properties of typical processes on $T_d$ and the properties of large random $d$-regular graphs.

Our main theorem is equivalent to the statement that if a typical process satisfies the eigenvector equation at every vertex of $T_d$ and has finite variance (at every vertex), then the process is jointly Gaussian. Note that such Gaussian eigenvector processes on $T_d$ are completely characterized and there is a unique one for each possible eigenvalue. A key ingredient in our proof is a general entropy inequality for typical processes. It implies that typical eigenvector processes obey another inequality involving differential entropy. From here we finish the proof using heat propagation on the space of typical eigenvector processes combined with DeBrujin's identity for Fisher information. Gaussianity will follow from the fact that heat propagation converges to a Gaussian distribution. 

A well studied subclass of typical processes is the class of factor of i.i.d.\ processes. These processes appeared first in ergodic theory but they are also relevant in probability theory, combinatorics, statistical physics and in computer science. Not every typical process is factor of i.i.d.; this follows from the results of Gamarnik and Sudan \cite{gamarnik}; see also Rahman and Vir\'ag \cite{mustazee}.
Despite of recent progress in the area \cite{cordec, balint, lewis, csoka, robin, damien, harangi, HLSz, russ, nazarov, fiidperc}, a satisfying understanding of factor of i.i.d.\ processes is only available in the case $d=2$ \cite{ornstein}, which is basically equivalent to the framework of classical ergodic theory of $\mathbb{Z}$ actions. Our results imply that if an invariant process (with finite variance) is in the weak closure of factor of i.i.d. processes and satisfies the eigenvector equation then the process is Gaussian. This answers a question of B. Vir\'ag.

{\bf Outline of the paper.} In Chapter \ref{chap:main}, we formulate the main results for finite random $d$-regular graphs. Chapter \ref{prelim} and \ref{chap:eigen} contain general statements about invariant processes, eigenvector processes, entropy and almost eigenvectors. Chapter \ref{typpro} provides the translation of our main result to the infinite setting using typical eigenvector processes on the infinite $d$-regular tree.  In Chapter \ref{chap:entr}, we prove a necessary condition  for a process to be typical in the form of an entropy inequality.  In Chapter \ref{smooth} we reduce the limiting form of the main theorem to a special family of eigenvector processes called smooth eigenvector processes. Chapter \ref{chap:entrtyp} gives a differential entropy inequality for smooth eigenvector processes. In Chapter \ref{chap:cov}, we calculate the eigenvalues of special submatrices of the covariance matrices of eigenvector processes corresponding to balls around vertices and edges on the tree. In Chapter \ref{chap:impr}, we use the results from Chapter \ref{chap:cov} to prove that among smooth typical eigenvector processes the Gaussian minimizes  the differential entropy formula from Chapter \ref{chap:entrtyp}.  On the other hand, in Chapter \ref{chap:heat}, we show that the Gaussian eigenvector process maximizes the same formula. Moreover, we finish the proof of the main result in Chapter \ref{chap:heat}.

\section{The main theorem}

\label{chap:main}
In this chapter we state the main theorem first in a simpler but weaker form and later in the strong form. 
An $\varepsilon$-almost eigenvector of a matrix $A\in\mathbb{R}^{n\times n}$ with eigenvalue $\lambda$ is a vector $v\in\mathbb{R}^n$ such that $\|v\|_2=1$ and $\|Av-\lambda v\|_2\leq\varepsilon$. 
To every vector $v$ in $\mathbb{R}^n$ we associate a probability distribution ${\rm distr}(v)$ on $\mathbb{R}$ obtained by choosing a uniform random entry from $v$. If $\|v\|_2=1$, then the second moment of ${\rm distr}(v)$ is $1/n$. Thus, to avoid degeneracy in this case, it is more natural to consider ${\rm distr}(\sqrt{n}v)$ whose second moment is $1$.  
We will compare probability distributions on $\mathbb{R}$ using an arbitrary but fixed metrization of the weak convergence of probability measures. Our theorem in a weak form says the following

\begin{theorem}[Weak form of main theorem] For every $\varepsilon>0$ there are constants $N, \delta$ such that if $G$ is a random $d$-regular graph on $n\geq N$ vertices, then with probability at least $1-\varepsilon$ the following holds. We have that every $\delta$-almost eigenvector $v$ of $G$ (with entry sum $0$) has the property that ${\rm distr}(\sqrt{n}v)$ is at most $\varepsilon$-far from some Gaussian distribution $N(0,\sigma)$ in the weak topology where $0\leq\sigma\leq 1$.
\end{theorem}

Note that if ${\rm distr}(\sqrt{n}v)$ is close to the degenerate distribution $N(0,0)$ then most of the $\ell^2$ weight of $v$ is concentrated on $o(n)$ points. Such vectors are called localized. In general, if $\sigma$ is smaller than, $1$ then some of the $\ell^2$ weight is concentrated on $o(n)$ vertices and the rest is Gaussian.

To formulate our main theorem in the strong form we need some more notation. Recall that $T_d$ denotes the infinite $d$-regular tree and $o$ is a distinguished vertex called ${\it root}$ in $T_d$. We will denote the vertex set $V(T_d)$ of $T_d$ by $V_d$.  For two vertices in a graph we write $v\sim w$ if they are connected to each other. Let $[n]$ denote the set $\{1,2,\dots,n\}$ and let $G$ be a $d$-regular graph on the vertex set $[n]$. We denote by ${\rm Hom}^*(T_d,G)$ the set of all covering maps from $T_d$ to $G$. In other words ${\rm Hom}^*(T_d,G)$ is the set of maps $\phi:V_d\to V(G)$ such that for every vertex $v\in V_d$ the neighbors of $v$ are mapped bijectively to the neighbors of $\phi(v)$. 
The set ${\rm Hom}^*(T_d,G)$ has a natural probability measure.  We first choose the image of $o$ uniformly at random in $V(G)$. Then we recursively extend the map $\phi$ to larger and larger neighborhoods of $o$ in a random (conditionally independent) way preserving the local bijectivity. It is easy to see that this probability distribution is independent from the choice of $o$.   

Let $X$ be a topological space and $f:[n]\to X$ be a function. We define the probability distribution ${\rm distr}^*(f,G)$ on $X^{V_d}$ as the distribution $f\circ \phi$ where $\phi$ is a random covering in ${\rm Hom}^*(T_d,G)$.
In other words ${\rm distr}^*(f,G)$ is a random lift of $f$ to $T_d$ using a random covering of $G$ with $T_d$. By regarding vectors $v\in\mathbb{R}^n$ as functions form $[n]$ to $\mathbb{R}$ it makes sense to use ${\rm distr}^*(v,G)$.

To formulate our main theorem we need the concepts of eigenvector processes and Gaussian waves on $T_d$ (see \cite{wave}). An {\em eigenvector process} with eigenvalue $\lambda$ is a joint distribution  $\{X_v\}_{v\in V_d}$ of real valued random variables with variance 1 such that it is ${\rm Aut}(T_d)$ invariant and satisfies the eigenvector equation 
\begin{equation}\label{eq:sv}\sum_{v\sim o} X_v=\lambda X_o\end{equation} with probability $1$. 
Note that the group invariance implies that the eigenvector equation is satisfied at every vertex on $T_d$. 
We call an eigenvector process {\em trivial} if $\mathbb E(X_o)\neq 0$. Notice that the triviality of an eigenvector process implies that $\lambda=d$. This follows by taking expectation in \eqref{eq:sv} and using the invariance of the process. Furthermore, trivial eigenvector processes are constants in the sense that $X_v=X_w$ holds with probability one for every pair of vertices $v,w$ in $V_d$.
A {\em Gaussian wave} is an eigenvector process whose joint distribution is Gaussian.  It is proved in \cite{wave} that  for every $-d\leq\lambda\leq d$ there is a unique Gaussian wave $\Psi_\lambda$ with eigenvalue $\lambda$.

Let us choose a fix metrization of the weak topology on $\mathbb{R}^{V_d}$. Our main theorem on random $d$-regular graphs is the following.

\begin{theorem}[Main theorem]\label{main} For every $\varepsilon>0$ there exist  constants $N, \delta$ such that if $G$ is a random $d$-regular graph on $n\geq N$ vertices then with probability at least $1-\varepsilon$ the following holds. For every $\delta$-almost eigenvector $v$ of $G$ (with entry sum $0$) has the property that ${\rm distr}^*(\sqrt{n}v,G)$ is at most $\varepsilon$-far (in the weak topology) from some Gaussian wave $\Psi_\lambda$ with $|\lambda|\leq 2\sqrt{d-1}$.
\end{theorem}

\begin{corollary} For every $\varepsilon>0$ and $k\in \mathbb N$ there exist  constants $N, \delta$ such that if $G$ is a random $d$-regular graph on $n\geq N$ vertices then with probability at least $1-\varepsilon$ the following holds. For every $k$-tuple of $\delta$-almost eigenvectors $Q=(v^{(1)}, v^{(2)}, \ldots, v^{(k)})$ of $G$ (with entry sum $0$) corresponding to eigenvalues $\lambda_1, \ldots, \lambda_k$ with the property $|\lambda_i-\lambda_j|<\delta$ we have that ${\rm distr}^*(\sqrt{n}Q, G)$ is at most $\varepsilon$-far (in the weak topology) from some Gaussian wave $\Psi_\lambda$ with $|\lambda|\leq 2\sqrt{d-1}$.

\end{corollary}

\section{Preliminaries}
\label{prelim}

\subsubsection*{Invariant processes}

For a separable metric space $Y$ we denote by $I_d(Y)$ the set of Borel probability measures on $Y^{V_d}$ that are invariant under the automorphisms of the tree. More precisely, for every $\tau\in \mathrm{Aut}(T_d)$ (not necessarily fixing the root), the probability measure on $Y^{V_d}$ is required to be invariant under the natural $Y^{V_d}\rightarrow Y^{V_d}$ map induced by $\tau$. Note that eigenvector processes are in $I_d(\mathbb{R})$. If $\mu\in I_d(Y)$ and $F\subseteq V_d$, then we denote by $\mu_F$ the marginal distribution of $\mu$ at $F$. We can equivalently think of $\mu\in I_d(Y)$ as a joint distribution $\{X_v\}_{v\in V_d}$ of $Y$-valued random variables that is invariant under the automorphism group of $T_d$. In this language $\mu_F$ is the same as the joint distribution $\{X_v\}_{v\in F}$. If both $F$ and $Y$ are finite then $\mu_F$ is a probability distribution on the finite set $Y^F$. In this case we denote the entropy of $\mu_F$ by $\mathbb{H}(F)$. By invariance of $\mu$ the quantity $\mathbb{H}(F)$ depends only on the isomorphism class of $F$. 


We will consider convergence in $I_d(Y)$ with respect to the weak topology. Since the weak topology is metrizable we can always choose a fixed metrization of it in advance.

\subsubsection*{Almost eigenvectors}

Let $A\in \mathbb R^{n\times n}$ be a matrix. An $\varepsilon$-almost eigenvector of $A$ (with eigenvalue $\lambda$) is a vector $v\in \mathbb R^n$ such that $\|v\|_2=1$ and $\|Av-\lambda v\|_2\leq \varepsilon$.

\begin{lemma} \label{lem:almost}Let $A$ be the adjacency matrix of a $d$-regular graph on $n$ vertices. Let $\lambda_2(A)$ denote the second largest (in absolute value) eigenvalue of $A$. Let $v$ be an $\varepsilon$-almost eigenvector with eigenvalue $\lambda$ such that the entry sum of $v$ is $0$. Then $|\lambda|\leq \lambda_2(A)+\varepsilon$.
\end{lemma}
\begin{proof}
Since $v$ has $0$ entry sum, we can write $v=\sum a_i v_i$, where each $v_i$ is a nonconstant eigenvector  of $A$ with eigenvalue $\lambda_i$. It follows that $(A-\lambda I)v=\sum a_i (\lambda_i-\lambda)v_i$, where $|\lambda_i|\leq \lambda_2(A)$. Thus, $\varepsilon^2\geq \|(A-\lambda I)v\|^2_2=\sum a_i^2 (\lambda_i-\lambda)^2$. Suppose that $|\lambda|\geq \lambda_2(A)$ (otherwise the statement is trivial). Then $|\lambda_i-\lambda|\geq |\lambda|-\lambda_2(A)$. Therefore $\varepsilon^2\geq (\sum a_i^2)(|\lambda|-\lambda_2(A))^2$, which completes the proof by using that $\sum a_i^2=\|v\|_2^2=1$. \hfill $\square$
\end{proof}

As we mentioned in the introduction, we will give examples for both localized and delocalized almost eigenvectors on essentially large girth $d$-regular graphs. (Note that a graph is called essentially large girth if most vertices are not contained in short cycles.) The purpose of these examples is to show that our results on the almost eigenvectors of random regular graphs are best possible in the sense that all $0\leq \sigma \leq 1$ can indeed occur in the statement. 

We will need some preparation. For $k\geq 1$ and $x\in [-1,1]$ let 
\begin{equation}\label{eq:fkx}f(k,x)=\frac{1}{\sqrt{d(d-1)^{k-1}}}q_k(x),\end{equation} where
\begin{equation}\label{eq:qkx}q_k(x)=\sqrt{\frac{d-1}{d}}U_k(x)-\frac{1}{\sqrt{d(d-1)}}U_{k-2}(x); \qquad U_k(\cos \vartheta)=\frac{\sin((k+1)\vartheta)}{\sin \vartheta}.\end{equation}
($U_k(x)$ is the Chebyshev polynomial of the second kind.)  Let $g_{\lambda}: V_d\rightarrow \mathbb R$ be the function defined by \[g_{\lambda}(v)=f(|v|, \lambda/(2\sqrt{d-1})),\] where $|v|$ denotes the distance of $v$ and the root $o$. It is easy to see (and well known) that $g_{\lambda}$ satisfies the eigenvector equation with eigenvalue $\lambda$ at every vertex $v$ (for  $\lambda\in [-2\sqrt{d-1}, 2\sqrt{d-1}]$). 

Now let $G$ be a $d$-regular graph such that there is a vertex $w\in V(G)$ with the property that the shortest cycle containing $w$ has length at least $2n$. We define the function $g'_{\lambda}: V(G)\rightarrow \mathbb R$ by $g'_{\lambda}(v)
=f(d(w,v), \lambda/{2\sqrt{d-1}})$ if $d(v,w)<n$ and $0$ otherwise. It is easy to see that $u=g'_{\lambda}/\|g'_{\lambda}\|_2$ is an almost eigenvector with eigenvalue $\lambda$ with error tending to $0$ as $n\rightarrow \infty$. Furthermore, if $|V(G)|$ is much larger than $n$, then $u$ is close to the constant $0$ distribution in the weak topology. Thus we obtain examples for completely localized almost eigenvectors for $d$-regular graphs (for all $\lambda\in [-2\sqrt{d-1}, 2\sqrt{d-1}]$), which corresponds to the case $\sigma=0$. 

We switch to the delocalized example. In \cite{harangi}, the authors construct eigenvector processes on $T_d$ for every $\lambda\in [-2\sqrt{d-1}, 2\sqrt{d-1}]$ that are weak limits of factor of i.i.d.\ processes. These processes have the property that they can be arbitrarily well approximated on any essential large girth $d$-regular graph. It is easy to see that these approximations are almost eigenvectors that are completely delocalized. This corresponds to the case $\sigma=1$. Finally, every $\sigma$ occurs by mixing completely localized and completely delocalized almost eigenvectors corresponding to the same $\lambda$.

\section{Eigenvectors and eigenvector processes on the tree}

\label{chap:eigen}
For a vertex set $F\subseteq V_d$ we denote by $B_k(F)$ the neighborhood of radius $k$ around $F$.
Let $F\subseteq V_d$ be a subset of the vertices of the tree, and let $f\in \mathbb R^{F}$. We say that $v$ satisfies the eigenvector equation with eigenvalue $\lambda$ if for every $v\in F$ with $B_1(v)\subseteq F$ we have that $\lambda f(v)=\sum_{w\sim v} f(w)$. It is clear that for a fixed $\lambda$ these vectors form a linear subspace of $\mathbb R^F$ that we denote by $W_{\lambda}(F)$. We will need a formula for $\dim\, W_{\lambda}(F)$ for a family of special finite sets $F$.

Given $F$, we say that $F_0\subseteq F$ is a basis if for all $f\in \mathbb R^{F_0}$ and $\lambda \in \mathbb R$ the subspace $W_{\lambda}(F)$ contains exactly one extension of $f$ to $F$. It is clear that if $F_0$ is a basis in $F$, then $\dim\, W_{\lambda}(F)=|F_0|$ for all $\lambda\in \mathbb R$. 

\begin{lemma}Let $F_0$ be a basis of a path-connected set $F$. Suppose that $v\in F$ and $|F\cap B_1(v)|=2$. Furthermore, let $D\subseteq B_1(v)$ such that $|D|=d-2$ and $D\cap F=\emptyset$. Then $F_0\cup D$ is a basis of $F\cup B_1(v)$. \label{lem:comb}
\end{lemma}
\begin{proof}
Let $f$ be a function from $F_0\cup D$ to $\mathbb R$. By assumption, we have that $f|_{F_0}$ extends to a unique function $\tilde f$ on $F$. Now using the eigenvector equation at $v$, we obtain a unique value for $B_1(v)\setminus (D \cup F)$. Note that the connectivity of $F$ implies that the function constructed this way is in $W_{\lambda}(F\cup B_1(v))$.
\end{proof}\hfill $\square$

 We will use $C$ to denote the star $B_1(o)$ and we will use $e$ to denote a distinguished edge in $T_d$.
It is clear that if $v$ is a neighbor of $o$, then $C\setminus \{v\}$ is a basis of $C$. Similarly $e$ is a basis of itself. Using Lemma \ref{lem:comb} and induction, we obtain that 
\begin{equation}\label{eq:dim} \dim\, W_{\lambda}(B_k(C))=|\partial B_k(C)|=d(d-1)^{k}; \qquad \dim\, W_{\lambda}(B_k(e))=|\partial B_k(e)|=2(d-1)^k \end{equation}
holds for every $\lambda \in \mathbb R$ and $k \in \mathbb N$, where $\partial F$ denotes the boundary of a set $F$. 

Recall that an invariant process $\{X_v\}_{v\in V_d}$ is an eigenvector process with eigenvalue $\lambda$ if it satisfies the eigenvector equation \eqref{eq:sv} with probability $1$ and $\mathrm{Var}(X_v)=1$. Notice that if  $F$ is any vertex set in $T_d$, then $\{X_v\}_{v\in F}$ is supported on $W_{\lambda}(F)$. In the rest of this chapter we investigate the joint distribution $\{X_v\}_{v\in S}$ in a non-trivial eigenvector process where $S$ is one of $B_k(C)$ or $B_k(e)$ for some $k\in\mathbb{N}$. Our goal is to find uncorrelated linear combinations of the variables $\{X_v\}_{v\in S}$ with the property that they linearly generate every random variable in $\{X_v\}_{v\in S}$. Observe that since the covariance matrix of $\{X_v\}_{v\in S}$ has rank at most $\dim(W_\lambda(S))=|\partial S|$, it is enough to find that many uncorrelated linear combinations with nontrivial variance to guarantee that they generate everything. 

Let $Y=(Y_1, Y_2,\dots,Y_n)$ be an $n$-dimensional distribution with $0$ mean such that the covariance matrix is of full rank. We can always find a linear transformation $T:\mathbb{R}^n\rightarrow\mathbb{R}^n$ with the property that the covariance matrix of $TY$ is the identity matrix. The probability distribution $TY$ is unique up to an orthogonal transformation on $\mathbb{R}^n$. We call it the standardized version of $Y$. 

Let $\mu\in I_d(\mathbb{R})$  represented by an invariant joint distribution $\{X_v\}_{v\in V_d}$ of random variables. Assume that $\mathbb{E}(X_o)=0$. With every directed edge $(v,w)$ of $T_d$ we associate a probability distribution on $\mathbb{R}^{d-2}$ defined up to an orthogonal transformation.  Let $\{v_i\}_{i=1}^{d-1}$ be the set of neighbors of $w$ different from $v$.  We denote by $A_{v,w}$ the standardized version of $\{X_{v_i}-X_{v_{d-1}}\}_{i=1}^{d-2}$. It is easy to see that (the orthogonal equivalence class of) $A_{v,w}$ does not depend on the labeling of the vectors $v_i$. 

We introduce the symmetric relation $r$ on directed edges of $T_d$ such that $((x,y),(v,w))\in r$ if and only if the unique shortest path connecting $y$ and $w$ contains at least one of $x$ and $v$. The next lemma implies that if $((x,y),(v,w))\in r$ then $A_{x,y}$ and $A_{v,w}$ are uncorrelated and so if $\mu$ is Gaussian then $A_{x,y}$ and $A_{v,w}$ are independent.

\begin{lemma}\label{arcuncorr} Let $(v,w)$ be a directed edge in $T_d$. Let $\{v_i\}_{i=1}^{d-1}$ be the set of neighbors of $w$ different from $v$. Let $X=\sum_{j=1}^k a_jX_{u_j}$ be a linear combination such that the shortest path connecting $w$ and $u_j$ contains $v$ for every $1\leq j\leq k$. Then $\mathbb{E}(X(X_{v_i}-X_{v_{d-1}}))=0$ holds for $1\leq i\leq d-2$.  
\end{lemma}

\begin{proof} The condition on the vertices $u_j$ guarantees that for every fix $j$ the distance of $u_j$ from $v_i$ does not depend on $i$. Using this and the automorphism invariance of $\mu$ it follows that $\mathbb{E}(X_{u_j}(X_{v_i}-X_{v_{d-1}}))=0$ holds for every $1\leq j\leq k$ and $1\leq i\leq d-2$. By linearity of expected value the proof is complete. \hfill $\square$
\end{proof}

Let $S\subset V_d$ be either $B_k(C)$ or $B_k(e)$ for some $k\in\mathbb{N}$. Assume that $\mu$ is an eigenvector process. Let $p\in S$ be such that it has distance at least one from the boundary $\partial S$. 
Let $D$ denote the set of directed edges $(v,w)$ inside $S$ with the following three properties: $(a)$ the unique shortest path connecting $p$ and $w$ contains $v$, $(b)$ $p\notin\{v,w\}$, $(c)$ $B_1(w)\subset S$. Let $B_p$ denote the joint distribution $(X_v-X_p)_{v\sim p}$. We denote by $\mathcal{Q}(S,p,\mu)$ the joint distribution of the random variables $\{A_{v,w}\}_{(v,w)\in D}$ and $B_p$. Lemma \ref{arcuncorr} implies that the components of $\mathcal{Q}(S,p,\mu)$ are uncorrelated multidimensional random variables. It is easy to see that if $|\lambda|<d$, then the correlation matrix of each such multidimensional random variable is of full rank. This implies that if $|\lambda|<d$, then $\mathcal{Q}(S, p, \mu)$ provides a linear basis for $\{X_v\}_{v\in S}$. By counting dimensions, this yields the following corollary. 

\begin{corollary}\label{cor:qsp}
Let either $S=B_k(C)$ or $S=B_k(e)$ and $\mu$ an eigenvector process with eigenvalue $|\lambda|<d$. Then we have that 
\[\langle \supp\ (\mu_S)\rangle_{\mathbb R}=W_{\lambda}(S).\]
\end{corollary}

\section{Typical processes and the limiting form of the main theorem}

\label{typpro}

Now we describe a limiting form of our main theorem using typical processes on $T_d$. Typical processes on $T_d$ were first introduced in \cite{BSz} to study the properties of random regular graphs via ergodic theory. 
Here we use a slightly different definition which extends the original notion to processes that take values in a separable, metrizable space $Y$. 

\begin{definition} Let $Y$ be separable, metrizable topological space and let $\mu$ be a Borel probability distribution on $Y^{V_d}$. We say that $\mu$ is a typical process if there is a growing sequence of natural numbers $\{n_i\}$ with the following property. Assume that $\{G_i\}_{i=1}^n$ is a random graph sequence whose elements $G_i$ are independently chosen random $d$-regular graphs on $n_i$ vertices. Then with probability $1$ there are maps $f_i:V(G_i)\to Y$ such that the distributions ${\rm distr}^*(f_i,G_i)$ are converging to $\mu$ in the weak topology. 
\end{definition}

Our main theorem in the limit setting is the following.

\begin{theorem}[Limiting form of the main theorem]\label{mainlim} If $\mu$ is a nontrivial typical eigenvector process with eigenvalue $\lambda$, then $|\lambda|\leq 2\sqrt{d-1}$ and $\mu$ is the Gaussian wave $\Psi_\lambda$.
\end{theorem}

Using the fact that weak limits of factor of i.i.d processes are typical (see \cite{BSz}) we obtain the next corollary which answers a question of B. Vir\'ag.
\begin{corollary}\label{mainlimcorr} If $\mu$ is a nontrivial eigenvector process that is a weak limit of factor of i.i.d.\ processes, then $\mu$ is a Gaussian wave with eigenvalue $|\lambda|\leq 2\sqrt{d-1}$. 
\end{corollary}
Note that Corollary \ref{mainlimcorr} implies that if many eigenvector processes corresponding to the same eigenvalue are coupled in a way that the coupling is a weak limit of factor of i.i.d.\ processes, then its distribution is jointly Gaussian.

We emphasize that the first part of the statement in Theorem \ref{mainlim}, namely that $|\lambda|\leq 2\sqrt{d-1}$ is a consequence of Friedman's theorem \cite{friedman}. We prove this implication in Lemma \ref{lem:equiv}. In addition, the main goal of this chapter is to show that Theorem \ref{mainlim} implies Theorem \ref{main}.

\begin{lemma}\label{lem:friedman}
If $\mu$ is a nontrivial typical eigenvector process with eigenvalue $\lambda$, then $|\lambda|\leq 2\sqrt{d-1}$.
\end{lemma}
\begin{proof}
Our first goal is to prove that there exists a sequence of $d$-regular graphs $\{G_i\}_{i=1}^{\infty}$ and vectors $\{f_i: V(G_i)\rightarrow \mathbb R\}_{i=1}^{\infty}$ such that $(a)$ $|\lambda_2(G_i)|\rightarrow 2\sqrt{d-1}$, where $\lambda_2(G)$ denotes the second largest (in absolute value) eigenvalue of a finite graph $G$; $(b)$ $\mathrm{distr}^*(f_i, G_i)\rightarrow \mu$ in the weak topology. Using that $\mu$ is typical, there exists a sequence $\{n_i\}_{i=1}^{\infty}$ of growing natural numbers such that with probability 1, if $\{G'_i\}_{i=1}^{\infty}$ is a sequence of independent random $d$-regular graphs with $|V(G'_i)|=n_i$, then there exists a sequence of functions $\{f'_i\}_{i=1}^{\infty}$ satisfying $(b)$.   Friedman's theorem \cite{friedman} implies that with probability 1 we can choose a further subsequence satisfying $(a)$.

We fix an arbitrary $\varepsilon>0$. Let $X_o$ be a random variable with distribution $\mu_o$. Since $\mathbb E(X^2_o)=1$, we can find $k>0$ such that $k$ is a continuity point of the cumulative distribution function of $X_o$ and $k^2\mathbb P(|X_o|>k)<\varepsilon$.  Furthermore, if $k$ is sufficiently large, we can also assume that $\mathbb E([X_o]^2_k)\geq 1/2$ and $|\mathbb E([X_o]_k)|\leq \varepsilon$,  where $[x]_k=\max(\min(x, k), -k)$. Using that $\mathrm{distr}^*(f_i, G_i)$ converges to the eigenvector process $\mu$, we obtain that if $i$ is large enough, then $(i)$ $k^2\mathbb P(\mathrm{distr}(|f_i|)>k)<2\varepsilon$; $(ii)$ $\mathbb P(\mathrm{distr}(|(G_i-\lambda I)f_i|)\geq \varepsilon)\leq \varepsilon/k^2$; $(iii)$ $\mathbb E(\mathrm{distr}([f_i]_k^2))\geq 1/3$; $(iv)$ $c_i=|\mathbb E(\mathrm{distr}([f_i]_k))|\leq 2\varepsilon$. Note that $(i)$ implies that $f_i\neq [f_i]_k$ holds on a vertex set of density at most $2\varepsilon/k^2$. It follows that $(G_i-\lambda I)[f_i]_k\neq (G_i-\lambda I)f_i$ holds on a vertex of density at most $(d+1)2\varepsilon/k^2$. Furthermore, we have that $\|(G_i-\lambda I)[f_i]_k\|_{\infty}\leq (d+|\lambda|)k$. Putting all this together, we obtain that for $i$ large enough
\[\|(G_i-\lambda I)([f_i]_k-c_i)\|_2\leq \sqrt{\varepsilon^2n_i+(2d+3)(\varepsilon/k^2)\cdot n_i(d+|\lambda|)^2k^2}+|d-\lambda|c_i\sqrt{n_i}\]\[\leq\sqrt{n_i}\big(\sqrt{\varepsilon^2+(2d+3)\varepsilon(d+|\lambda|)^2}+2|d-\lambda|\varepsilon\big).\]
Let $v_i=([f_i]_{k}-c_i)/\|[f_i]_{k}-c_i\|_2$. Using that $(iii)$ implies that $\|[f_i]_k\|_2^2\geq 1/3n_i$, we obtain that for an appropriate choice of small enough $\varepsilon$  and large enough $i$  the quantity $\|(G_i-\lambda I)v_i\|_2$ is arbitrarily small.
Thus, by using $(a)$ and Lemma \ref{lem:almost}, we get that $|\lambda|\leq 2\sqrt{d-1}$. 
\hfill $\square$
\end{proof}

\begin{proposition} \label{lem:equiv} Theorem \ref{mainlim} implies Theorem \ref{main}.
\end{proposition}

We need a few notions and lemmas. Let $\mathcal{P}$ denote the set of Borel probability distributions $\mu$ on $\mathbb{R}^{V_d}$ which have a second moment bounded from above by $1$ at each coordinate. 
By tightness of $\mathcal P$, we have that $\mathcal{P}$ is compact with respect to the weak topology of measures.  Let $m$ be a fixed metrization of the weak topology on $\mathcal{P}$.  Let us denote by $\mathcal{T}$ the set of closed subsets in $\mathcal{P}$, and by $d_H$ the  the Hausdorff metric on $\mathcal{T}$. We have that $d_H$ induces a compact topology on $\mathcal{T}$.
Let us define the distance $m^*$ for $d$-regular graphs in the following way. If $G_1$ and $G_2$ are $d$-regular graphs then $m^*(G_1,G_2)$ is the infimum of the numbers $\delta$ with the property the if $f_1:V(G_1)\to \mathbb{R},f_2:V(G_2)\to \mathbb{R}$ are arbitrary functions with $\mathbb{E}({\rm distr}(f_1^2)),\mathbb{E}({\rm distr}(f_2^2))\leq 1$ then there are functions $f_1':V(G_2)\to\mathbb{R}, f_2':V(G_1)\to \mathbb{R}$ with $\mathbb{E}({\rm distr}(f_1'^2)),\mathbb{E}({\rm distr}(f_2'^2))\leq 1$ such that $$m({\rm distr}^*(f_1,G_1),{\rm distr}^*(f_1',G_2))\leq\delta ~~~,~~~m({\rm distr}^*(f_2,G_2),{\rm distr}^*(f_2',G_1))\leq\delta.$$
We can describe the metric $m^*$ in terms of the metric $d_H$ as follows. For a graph $G$, let 
\[S(G)=\{\mathrm{distr}^*(f, G)|f\in \mathbb{R}^{V(G)},\mathbb{E}(\mathbb{\rm distr}(f^2))\leq 1\}.\]
We have that $m^*(G_1, G_2)=d_H(S(G_1), S(G_2))$.

\begin{definition} We say that a finite $d$-regular graph $G$ is $\varepsilon$-typical for some $\varepsilon>0$ if with probability at least $1-\varepsilon$ a random $d$-regular graph  $G'$ on $|V(G)|$ vertices  has the property that $m^*(G, G')<\varepsilon$.
\end{definition}

\begin{lemma} \label{lem:lip}For every $\varepsilon>0$ and $T\in \mathcal T$ there exists $n(\varepsilon)$ such that for all $N>n(\varepsilon)$ there exists a value $c$ satisfying 
\begin{equation}\label{eq:dh}\mathbb P(|d_H(S(G), T)-c|>\varepsilon)<\varepsilon.\end{equation}
\end{lemma}
\begin{proof}The proof relies on a certain continuity property of the metric $m^*$ with respect to small changes in a graph.  More precisely, we show that for every $\varepsilon_2$ if $N$ is large enough, then for every pair $G, G'$ of $d$-regular graphs on the vertex set $[N]$ satisfying $|E(G)\Delta E(G')|\leq 4$ we have that $m^*(G, G')<\varepsilon_2$. The significance of the number $4$ comes from the fact that $d$-regular graphs can be transformed into each other through a sequence of operations in which two independent edges $(u_1, v_1), (u_2, v_2)$ are replaced by $(u_1, v_2), (u_2, v_1)$. To prove the continuity property, we show that if $N$ is large enough, then the inequality 
\[m(\mathrm{distr}^*(f, G), \mathrm{distr}^*(f, G'))<\varepsilon_2\]
holds for every function $f: [N]\rightarrow \mathbb{R}$ with $\mathbb{E}(f^2)\leq 1$. Let $k$ be an arbitrary integer. Observe that the marginal distribution $\mathrm{distr}^*(f, G)|_{B_k(o)}$ can be obtained from the distribution of $\mathbb{R}$-colored neighborhoods of radius $k$ of a random vertex in $G$ (and the analogous statement holds for $G'$). Since $|E(G)\Delta E(G')|$ intersects such a neighborhood with probability tending to 0 as $|V(G)|\rightarrow \infty$, we have that the distance between $\mathrm{distr}^*(f, G)|_{B_k(o)}$ and $\mathrm{distr}^*(f, G')|_{B_k(o)}$ converges to 0 in any metrization of the weak topology of $\mathbb{R}^{B_k(o)}$. This implies the desired continuity property. 

We obtain by the above statement and the triangle inequality that if $|E(G)\Delta E(G')|\leq 4$ and $N$  is large enough, then $|d_H(S(G), T)-d_H(S(G'), T)|<\varepsilon_2$.
It is well-known that  graph parameters on random $d$-regular graphs that satisfy this Lipschitz property  are concentrated around their mean; see \cite[Theorem 2.19]{wormald}. 
\end{proof} \hfill $\square$

\begin{lemma}\label{lem:typgraph}
For every $\varepsilon>0$ there exists $n(\varepsilon)$ such that for every $N>n(\varepsilon)$ with probability at least $1-\varepsilon$ a random $d$-regular graph on $N$ vertices is $\varepsilon$-typical.
\end{lemma}

\begin{proof}Let $M$ be a finite $\varepsilon/2$-net in $\mathcal T$. If $G$ is a random $d$-regular graph on $N$ vertices, then there exists $T\in M$ with the property that 
\begin{equation}\label{eq:ph}\mathbb P(d_H(S(G), T)<\varepsilon/2)\geq 1/|M|. \end{equation}
We apply Lemma \ref{lem:lip} with $\varepsilon'\leq \varepsilon/4$. Combining inequalities \eqref{eq:dh} and \eqref{eq:ph}, we obtain that for $\varepsilon'<1/|M|$ and $N$ large enough $|c|\leq 3/4\varepsilon$. Then applying \eqref{eq:dh} again, the proof is complete.
\end{proof} \hfill $\square$

Now we enter the proof of Lemma \ref{lem:equiv}.

\begin{proof} We go by contradiction. If Theorem \ref{main} fails then there is a growing sequence of natural numbers $\{n_i\}_{i=1}^\infty$, $\varepsilon>0$  and a sequence $\{\delta_i\}_{i=1}^\infty$ with $\lim_{i\to\infty}\delta_i=0$ such that the following holds. If $G$ is a random $d$-regular graph on $n_i$ vertices then we have with probability at least $\varepsilon$ that there is an $\delta_i$-almost eigenvector $v$ of $G$ (with entry sum $0$) such that ${\rm distr}^*(v,G)$ is at least $\varepsilon$-separated from any Gaussian wave in the weak topology. 

From Lemma \ref{lem:typgraph} we obtain that there is a sequence $\{\varepsilon'_i\}_{i=1}^\infty$ with $\lim_{i\rightarrow \infty} \varepsilon_i'=0$ such that a random $d$-regular graph on $n_i$ vertices is $\varepsilon_i'$-typical with probability at least $1-\varepsilon_i'$ for every $i$. There exists an index $j$ such that for all $i\geq j$ we have $\varepsilon_i'<\varepsilon$. This implies that for all $i\geq j$ we can choose a graph $G_i$ on $n_i$ vertices such that $G_i$ is $\varepsilon_i'$-typical and there exists a $\delta_i$-almost eigenvector $f_i$ of $G_i$ (with entry sum 0) satisfying that ${\rm distr}^*(\sqrt{n_i}f_i, G_i)$ is at least $\varepsilon$-separated from any Gaussian wave in the weak topology. 

By choosing a subsequence we can assume (by abusing the notation) that ${\rm distr}^*(\sqrt{n_i}f_i,G_i)$ weakly converges to some measure $\mu\in\mathcal{P}$. It is clear that $\mu$ is a nontrivial eigenvector process which is at least $\varepsilon$-separated from any Gaussian wave in the weak topology. To get a contradiction it remains to show that $\mu$ is typical.

Again by choosing a subsequence we can assume that $\sum_{i=1}^\infty\varepsilon'_i<\infty$. Let $\{G'_i\}_{i=1}^\infty$ be such that $G_i'$ is a random $d$-regular graph on $n_i$ vertices and the terms of the sequence are independent. It follows from the Borel--Cantelli lemma that almost surely all but finitely many indices $i$ satisfy that $m^*(G_i',G_i)\leq\varepsilon'_i$. For such indices we can find $f_i'$ with $m({\rm distr}^*(f_i',G'_i),{\rm distr}^*(\sqrt{n_i}f_i,G_i))\leq\varepsilon'_i$. We obtain that ${\rm distr}^*(f_i',G_i')$ converges to $\mu$ showing that $\mu$ is typical.
\end{proof} \hfill $\square$

\section{Entropy inequality for typical processes}

\label{chap:entr}
Let $X$ be a separable metric space; let $F$ be a finite set and let $\mathcal{P}(F)$ denote the set of probability distributions on $F$ equipped with the topology generated by total variation distance. A continuous discretization of $X$ is a continuous function $\phi:X\rightarrow\mathcal{P}(F)$.
If $\alpha\in X^V$ is an $X$-coloring of a finite or countable set $V$, then we denote by $\phi*\alpha$ the probability distribution on $F^V$ obtained by independently choosing an element from $F$ for each $v\in V$ with distribution $\phi(\alpha(v))$.
If $\mu$ is a probability distribution on $X^V$, then we denote by $\phi*\mu$ the probability distribution obtained by first taking a $\mu$ random element $\alpha:V\rightarrow X$ and then in a second round of randomization we take $\phi*\alpha$.
The main result of this chapter is the next entropy inequality for typical processes.

\begin{theorem}\label{discmain} If $\mu\in I_d(X)$ is a typical process and $\phi:X\rightarrow\mathcal{P}(F)$ is a continuous discretization, then the process $\phi*\mu$ satisfies the next entropy inequality.
$$\mathbb{H}(B_k(C))-(d/2)\mathbb{H}(B_k(e))\geq\mathbb{E}_{\mu_o}(\mathbb{H}(\phi(x))).$$
\end{theorem}

Before proving Theorem \ref{discmain} we need some preparation.  Let us fix a metrization of the weak topology on $I_d(X)$. For a finite $d$-regular graph $G$ let $\omega(G)$ denote the infimum of the numbers $\varepsilon>0$ for which it is true that at least $1-\varepsilon$ fraction of the vertices of $G$ are not contained in a cycle of length at most $\lfloor 1/\varepsilon\rfloor$. The quantity $\omega(G)$ measures how similar the graph $G$ is to the tree $T_d$ in the Benjamini--Schramm metric. Throughout this chapter, $G$ is always assumed to be finite.

The operator ${\rm distr}^*$ maps $X$-colored $d$-regular graphs $(\alpha\in X^{V(G)},G)$ to invariant processes in $I_d(X)$. Using this correspondence and the metric on $I_d(X)$ we define the distance of an $X$-colored graph $(\alpha,G)$ and a process $\mu\in I_d(X)$ as $\omega(G)$ plus the distance of ${\rm distr}^*(\alpha,G)$ and $\mu$ in $I_d$. Note that if $(\alpha\in X^{V(G)},G)$ is an $X$-colored $d$-regular graph, then $(\phi*\alpha,G)$ is a probability distribution on $F$-valued colorings of the vertices of $G$. In this case ${\rm distr}^*(\phi*\alpha,G)$ is a probability distribution on $I_d(X)$, while $\phi*{\rm distr}^*(\alpha,G)$ is a single element in $I_d(X)$.

\begin{proposition}\label{discapprox} Let $\mu\in I_d(X)$ be an invariant process and $\phi$ is a continuous discretization of $X$. Then for every $\varepsilon>0$ there is $\delta>0$ such that if a colored $d$-regular graph $(\alpha\in X^{V(G)},G)$ is at most of distance $\delta$ from $\mu$, then with probability at least $1-\varepsilon$ we have that  $(\phi*\alpha,G)$ is at most of distance $\varepsilon$ from $\phi*\mu$.
\end{proposition}

The proof of Proposition \ref{discapprox} relies on the next technical lemma.

\begin{lemma}\label{discapproxl} Let $\mu\in I_d(X)$ be an invariant process. Let $r\in\mathbb{N}$,  $B=B_r(o)\subset V_d$ and let $\beta:F^B\rightarrow\mathbb{R}$ be any automorphism invariant function. Then for every $\varepsilon>0$ there is $\delta>0$ such that if a colored $d$-regular graph $(\alpha\in X^{V(G)},G)$ is at most of distance $\delta$ from $\mu$ then with probability at least $1-\varepsilon$ we have that 
$$|W|^{-1}\sum_{v\in W} \beta(\gamma|_{B_r(v)})$$ is at most $\varepsilon$-far from $\mathbb{E}(\beta(\phi*\mu_B))$ where $\gamma=\phi*\alpha$ and $W=\{v: v\in V(G), B_r(v)\simeq B\}$.
\end{lemma}

\begin{proof} Assume first that $(\alpha\in X^{V(G)},G)$ is an arbitrary $d$-regular $X$-colored graph whose distance is $\delta'$ from $\mu$ and let $\gamma,W$ as in the statement of the lemma. For a vertex $v\in W$ let $Y_v$ be the random variable with value $\beta(\gamma|_{B_r(v)})$.  Let $g: X^B\rightarrow\mathbb{R}$ be the function defined by $g(h)=\mathbb{E}(\beta(\phi*h))$. We have for $v\in W$ that $\mathbb{E}(Y_v)$ is equal to $g(\alpha|_{B_r(v)})$. Let $Y=|W|^{-1}\sum_{v\in W}Y_v$. It follows that 
\begin{equation}\label{tpei1} \mathbb{E}(Y)=|W|^{-1}\sum_{v\in W} g(\alpha|_{B_r(v)})=\int_{X^B}g~ d\nu_G,
\end{equation}
where $\nu_G$ describes the probability distribution of the isomorphism classes of $\alpha|_{B_r(v)}$ where $v\in W$ is a uniform random point.  Using the fact that $g$ is a continuous function we obtain that if $\delta'$ is small enough then the right hand side of (\ref{tpei1}) is at most $\varepsilon/2$ far from $\int_{X^B}~g~d\mu_B=\mathbb{E}(\beta(\phi*\mu_B))$.  Observe that $Y_v$ and $Y_w$ are independent if $v$ and $w$ have distance at least $2r+1$ in $G$. It follows that there are at most $|B_{2r+1}(o)||W|$ correlated pairs in $\{Y_v\}_{v\in W}$.  We obtain that the variance of $Y$ is at most $|B_{2r+1}(o)|^{1/2}|W|^{-1/2}\max |\beta|$. We use that $\omega(G)$ goes to $0$ as $\delta'$ goes to $0$ and thus $|W|$ tends to infinity. This implies that if $\delta'$ is sufficiently small then the variance of $Y$ is at most $\varepsilon^2/3$. Now by Chebyshev's inequality we have that $\mathbb{P}(|Y-\mathbb{E}(Y)|\geq \varepsilon/2)\leq \varepsilon^2.$ It follows that $\mathbb{P}(|Y-\mathbb{E}(\beta(\phi*\mu_B))|\geq\varepsilon)\leq \varepsilon^2$ which completes the proof. \hfill $\square$
\end{proof}

We continue with the proof of Proposition \ref{discapprox}.

\begin{proof} Let $\delta'$ be an arbitrary positive number. Let $(\alpha\in X^{V(G)},G)$ be an $X$-colored $d$-regular graph of distance $\delta'$ from $\mu$.  Let $(\gamma\in F^{V(G)},G)$ be chosen according to the probability distribution $(\phi*\alpha,G)$. Let $\varepsilon'>0$, $r=\lfloor 1/\varepsilon'\rfloor$ and  $W=\{v:v\in V(G), B_r(v)\simeq B\}$. It follows from Lemma \ref{discapproxl} that there is $c=c(\varepsilon')>0$ such that if $\delta'<c$ then the condition of Lemma \ref{discapproxl} holds for $(\gamma,G)$ simultaneously for every $0-1$ valued $\beta$ with probability at least $1-\varepsilon'$. (Here we use the fact that there are finitely many such functions $\beta$.)
If $c$ is small enough, then it also guarantees that $|W|/|V(G)|\geq 1-\varepsilon'$. Now it is clear that if $\varepsilon'$ is small enough, then these properties imply that $(\gamma,G)$ is at most $\varepsilon$ far from $\phi*\mu$.
\end{proof} \hfill $\square$

\medskip

For the next lemma let $R(d,n)$ denote the number of $d$-regular graphs on the vertex set $[n]$. In case $d$ is odd we will always assume that the number of vertices is even.

\begin{lemma}\label{countcolor} Let $F$ be a finite set and $\mu\in I_d(F)$.  Let $N(n,\varepsilon)$ denote the number of $d$-regular $F$-colored graphs on the vertex set $[n]$ whose distance from $\mu$ is at most $\varepsilon$. Assume that $\{n_i\}_{i=1}^\infty$ is a growing sequence of natural numbers. Then for every $\varepsilon>0$ and $k\in\mathbb{N}$ we have that 
\begin{equation}\label{cceq}
\mathbb{H}(B_k(C))-(d/2)\mathbb{H}(B_k(e))\geq \lim_{\varepsilon\to 0}\Bigl(\limsup_{i\to\infty} n_i^{-1}\log(N(n_i,\varepsilon)/R(d,n_i))\Bigr).
\end{equation}
\end{lemma}

\begin{proof} First we prove the statement for $k=0$. In this case we use a formula from \cite{BSz} that approximates the number $N'(n,\varepsilon)$ of $F$-colored graphs on $[n]$ in which the statistics of colored $1$-neighborhoods is $\varepsilon$-close to $\mu_C$. We have that
$$N'(n,\varepsilon)=R(d,n)\mathbb{H}(C)^{(n(1+o(1))}\mathbb{H}(e)^{-(dn/2)(1+o(1))}$$ where $o(1)$ is a quantity which goes to $0$ when first $n\to\infty$ and then $\varepsilon\to 0$. This implies that the right hand side of (\ref{cceq})  is equal to the left hand side when $N(n_i,\varepsilon)$ is replaced by $N'(n_i,\varepsilon)$. Now we use that for every $\varepsilon>0$ there exists $\varepsilon'>0$ such that $N'(n_i,\varepsilon)\geq N(n_i,\varepsilon')$ holds for all $i$. This finishes the proof of the first part.

The idea of the proof in case of $k>0$ is to generate a new process from $\mu$ in which the color of every vertex $v\in V_d$ is replaced by the isomorphism type of the colored neighborhood of $v$ of radius $k$. The main difficulty in this approach is that the isomorphism type describes the neighborhood only up to automorphisms which leads to extra constants in the entropy formulas. To control these constants (and to eventually get rid of them) we add some extra randomness to the process.

Let us introduce new processes  $\mu_{r,m}$ on $T_d$ for every $k,m\in\mathbb{N}$. If $r=0$, then $\mu_{0,m}$ denotes the $F\times [m]$ valued process in which we generate a $\mu$-random coloring on $T_d$ and then we add a second coordinate from $[m]$ to every vertex independently and uniformly. In general $\mu_{r,m}$ denotes the process obtained from $\mu_{0,m}$ by coloring $v\in V_d$ with the isomorphism class of the coloring of $B_r(v)$ in $\mu_{0,m}$. Let $a_{r,m,k}$ denote the left hand side (resp. $b_{r,m}$ denote the right hand side) of (\ref{cceq}) evaluated for $\mu_{r,m}$ and $k$. It is easy to see that $b_{r,m}$ does not depend on $r$. On the other hand the independence of the two coordinates implies that $b_{0,m}=b_{0,1}+\log m$. All together this means that $b_{r,m}=b_{0,1}+\log m$. It is clear that $a_{0,m,k}=a_{0,1,k}+\log m$. From the case $k=0$ we have that $a_{r,m,0}\geq b_{r,m}$. We can write this as $$ a_{0,1,r}+(a_{0,m,r}-a_{0,1,r})+(a_{r,m,0}-a_{0,m,r})\geq b_{0,1}+\log m$$ and thus $a_{0,1,r}+c_{r,m}\geq b_{0,1}$ holds for every $m$ where $c_{r,m}=a_{r,m,0}-a_{0,m,r}$. Since the inequality $a_{0,1,r}\geq b_{0,1}$ is equivalent to the statement of the lemma for $k=r$ it remains to show that $\lim_{m\to\infty} c_{r,m}=0$ holds for every $r$.

Let $t_r$ denote the size of the automorphism group of the rooted $d-1$-regular tree of depth $r$.
We claim that $\mathbb{H}_{\mu_{0,m}} (B_r(C))-\mathbb{H}_{\mu_{r,m}}(C)=d\log t_r + o(1)$ and that $\mathbb{H}_{\mu_{0,m}} (B_r(e))-\mathbb{H}_{\mu_{r,m}}(e)=2\log t_r + o(1)$ as $m\to\infty$. It is clear that this claim implies $c_{r,m}=o(1)$. We show the proof of the first claim. (The proof of the second one is almost identical.)
If $m$ is large enough then in the process $\mu_{0,m}$ restricted to $B_{r+1}(o)$ all labels are different with probability converging to $1$. In such a case knowing the isomorphism classes of the colored neighborhoods of radius $r-1$ of vertices in $C$ is equivalent with knowing the colored version of $B_{r+1}(o)$ up to an isomorphism that fixes $C$. The stabilizer of $C$ in the automorphism group of $B_{r+1}(o)$ is the $d$-th power of the automorphism group of the rooted $d-1$ regular tree of depth $r$. Thus the entropy loss of $\mathbb{H}_{\mu_{r,m}}(C)$ compared to $\mathbb{H}_{\mu_{0,m}} (B_r(C))$ is converging to $d\log t_r$.  \hfill $\square$
\end{proof}

Now we arrived to the proof of Theorem \ref{discmain}.

\begin{proof} According to Lemma \ref{countcolor} it is enough to show that the process $\phi*\mu$ satisfies 
$$\lim_{\varepsilon\to 0}\Bigl(\limsup_{i\to\infty} n_i^{-1}\log(N(n_i,\varepsilon)/R(d,n))\Bigr)\geq\mathbb{E}_{\mu_o}(\mathbb{H}(\phi(x)))$$ for some growing sequence $\{n_i\}_{i=1}^\infty$ of natural numbers.
For $n\in\mathbb{N},\varepsilon>0$ let $a(n,\varepsilon)$ denote the number of $d$-regular graphs $G$ on $[n]$ with the property that there exists an $X$-coloring $\alpha$ of $[n]$ such that $(G,\alpha)$ is of distance at most $\varepsilon$ from $\mu$. The fact that $\mu$ is typical is equivalent to the fact that there is a sequence $\{n_i\}_{i=1}^\infty$ such that $\lim_{i\to\infty}a(n_i,\varepsilon)/R(d,n_i)=1$ holds for every $\varepsilon>0$. 

From Proposition \ref{discapprox} we obtain that for every $\varepsilon_2>0$ there is $\varepsilon>0$ such that if a graph $G$ has an $X$ coloring $\alpha$ of distance at most $\varepsilon$ from $\mu$, then with probability at least $1-\varepsilon_2$ we have that $(\phi*\alpha,G)$ is of distance $\varepsilon_2$ from $\phi*\mu$. It follows that $G$ has at least $\exp(\mathbb{H}(\phi*\alpha,G)+o(1))$ 
$F$-colorings of distance at most $\varepsilon_2$ from $\phi*\mu$. On the other hand, we have $|V(G)|^{-1}\mathbb{H}(\phi*\alpha,G)=|V(G)|^{-1}\sum_{v\in V(G)}\mathbb{H}(\phi(\alpha(v)))$, which converges to $\mathbb{E}_{\mu_o}(\mathbb{H}(\phi(x)))$ as $\alpha$ converges to $\mu$. 

New let $N(n,\varepsilon)$ defined as in Lemma \ref{countcolor} for the process $\phi*\mu$. From the above observations we obtain that $n_i^{-1}\log (N(n_i,\varepsilon)/R(d,n_i))$ can be estimated from below by $\mathbb{E}_{\mu_o}(\mathbb{H}(\phi(x)))-o(1)$ as $n_i$ goes to infinity. Then Lemma \ref{countcolor} finishes the proof. \hfill $\square$
\end{proof}

\section{Smooth eigenvector processes}

\label{smooth}


Let $\mu$ be an eigenvector process. If $F\subseteq  V_d$ is a finite set  then the distribution of $\mu$, when restricted to $F$, is concentrated on the subspace $W_{\lambda}(F)$ (recall Chapter \ref{chap:eigen}). We denote by $\mathbb D_{\rm sp}(F,\mu)$ the differential entropy of $\mu_F$ measured inside this subspace using the Euclidean structure inherited from $\mathbb{R}^F$. We say that $\mu$ is {\it smooth} if $\mathbb{D}_{\rm sp}(B_k(C))$ and $\mathbb{D}_{\rm sp}(B_k(e))$ are finite for every $k$. In this chapter we reduce Theorem \ref{mainlim} to smooth eigenvector processes. The reduction will rely on the following statement.

\begin{proposition}\label{prop:masodik}
Let $\{X_v\}_{v\in V_d}$ be a typical eigenvector process with eigenvalue $\lambda$ and $\{Y_v\}_{v\in V_d}$ the unique Gaussian wave $\Psi_{\lambda}$ with eigenvalue $\lambda$. Then the independent sum $\{X_v+aY_v\}_{v\in V_d}$ is smooth for $a>0$.
\end{proposition}
\begin{proof}
Let $S$ be one of $B_k(C)$ and $B_k(e)$. 
Both $\{X_v\}_{v\in S}$ and $\{Y_v\}_{v\in S}$ are supported on $W_{\lambda}(S)$. 
Moreover, by Corollary \ref{cor:qsp}, we obtain that the support of $\Psi_{\lambda}|_S$ is equal to $W_{\lambda}(S)$. Using Lemma \ref{lem:veges} inside the space $W_{\lambda}(S)$, we get that the differential entropy of the joint distribution $\{X_v+aY_v\}_{v\in S}$ is finite. \hfill $\square$

\end{proof}

Assume that Theorem \ref{mainlim} holds for smooth typical eigenvector processes. Using Proposition \ref{prop:masodik} we obtain that if $\{X_v\}_{v\in V_d}$ is an arbitrary typical eigenvector process then
$\{X_v+aY_v\}_{v\in V_d}$ is smooth for all $a>0$. In addition, by Proposition \ref{prop:elso} and Proposition \ref{prop:hv}, $\{X_v+aY_v\}_{v\in V_d}$ is typical.  By our assumption, we obtain  the Gaussianity of $\{X_v+aY_v\}_{v\in V_d}$ for all $a>0$.  This implies the Gaussianity of $\{X_v\}_{v\in V_d}$ by  going to $0$ with $a$. 

\medskip

\section{Entropy Inequality for typical eigenvector processes}

\label{chap:entrtyp}

In this chapter we prove the next theorem.

\begin{theorem}\label{mainineq} Let $\mu\in I_d(\mathbb{R})$ be a smooth typical eigenvector process. Then
$$\mathbb D_{\rm sp}(B_k(C))-\frac{d}{2}\mathbb D_{\rm sp}(B_k(e))\geq 0$$ holds for every $k\geq 0$ integer.
\end{theorem}

To prove the above theorem we will need some preparation. 
For $a\in\mathbb{N}$ let us define the continuous discretization $t_{0,a}$ of $\mathbb{R}$ in the following way.
If $x>a$ (resp. $x<-a$), then $t_{0,a}(x)=a$ (resp. $t_{0,a}(x)=-a$) with probability $1$. Otherwise let $t_{0,a}(x)$ denote the probability distribution that takes $\lfloor x*a\rfloor/a$ with probability $1+\lfloor x*a\rfloor-x*a$ and takes $1/a+\lfloor x*a\rfloor/a$ with probability $x*a-\lfloor x*a\rfloor$. For $\sigma>0$ we define the discretization $t_{\sigma,a}$ by $t_{\sigma,a}(x)=t_{0,a}(x+\sigma N)$ where $N$ is a random variable with standard normal distribution. We denote by $t_{\sigma,a}^n$ the continuous discretization of $\mathbb{R}^n$ obtained by the coordinatewise independent application of $t_{\sigma,a}$.

\begin{lemma}\label{finombecsles} Let $X$ be a random variable with values in $\mathbb{R}^n$ with finite variance, i.e.\ $\mathbb{E}(\|X\|_2^2)<\infty$. Then we have for every fixed $\sigma>0$ that
$$\mathbb{H}(t_{\sigma,a}^n(X))=n\log a+\mathbb D(X+\sigma M)+o(1)$$
as $a\to\infty$, where $M$ is independent of $X$ and has standard normal distribution on $\mathbb{R}^n$.
\end{lemma}
The main difficulty of the proof of Lemma \ref{finombecsles} comes from the fact that the support of $X$ is not necessarily compact. We have to treat a situation where we refine and increase the interval of discretization simultaneously.

\begin{proof} By Lemma \ref{lem:veges}, the finite variance of $X$ guarantees that $\mathbb D(X+\sigma M)$ exists and is a finite quantity. Let $S_a=\{r/a | r\in\mathbb{Z}^n, \|r\|_\infty\leq a^2\}$ and let $S'_a=\{x|x\in S_a, \|x\|_\infty<a\}$. For $x\in S_a$ let $p_a(x)$ denote the probability of $x$ in the distribution $t_{\sigma,a}^n(X)$. We have that $$\mathbb{H}(t_{\sigma,a}^n(X))=\sum_{x\in S_a}-p_a(x)\log p_a(x).$$

Let $q_a$ denote the quantity $\mathbb{P}(t^n_{\sigma, a}(X)\in S_a\setminus S'_a)=\sum_{x\in S_a\setminus S'_a} p_a(x)$. The construction of $t_{\sigma,a}^n$ shows that $$q_a\leq \mathbb{P}(\|X+\sigma M\|_\infty\geq a-a^{-1}).$$
By Chebyshev's inequality we have that $$\mathbb{P}(\|X+\sigma M\|_\infty\geq a-a^{-1})=O(a^{-2})$$ and so $q_a=O(a^{-2})$. It follows that \[\sum_{x\in S_a\setminus S'_a}-p_a(x)\log p_a(x)\leq -q_a\log q_a+q_a\log |S_a\setminus S'_a|\leq -q_a\log q_a + q_a\log |S_a|\]\[=-q_a\log q_a+q_a n \log (2a^2+1)=o(1).\]

For $x=(x_1 ,x_2, \dots, x_n)\in\mathbb{R}^n$ let $g_a(x)=\prod_{i=1}^n\max\{1-a|x_i|,0\}$.  For every $x\in S'_a$ we have that $p_a(x)=\int_{z\in\mathbb{R}^n}g_a(z)f(x-z)$ where $f$ is the density function of $X+\sigma M$ on $\mathbb{R}^n$. Using that $a^n\int_{\mathbb{R}^n} g_a=1$ we have that $a^np_a(x)$ is a weighted average of the values of $f$ in an $L^\infty$-ball of radius $1/a$. It follows that for every $x\in S'_a$ there is $\alpha(x)\in\mathbb{R}^n$ with the property that $\|x-\alpha(x)\|_\infty\leq 1/a$ and $a^np_a(x)=f(\alpha(x))$ (using that $f$ is continuous). Now we have that
\begin{equation}\label{eqa1}
\sum_{x\in S'_a} -p_a(x)\log p_a(x)=(n\log a)\sum_{x\in S'_a}p_a(x)-a^{-n}\sum_{x\in S'_a} f(\alpha(x))\log f(\alpha(x)).
\end{equation}
It follows from $q_a=O(a^{-2})$ that
$$(n\log a)\sum_{x\in S'}p_a(x)=n(\log a)(1-q_a)=n\log a + o(1).$$ 
It remains to bound the second part of (\ref{eqa1}).
From the equation
\begin{equation}\label{eqa2}
\int_{z\in[-a+a^{-1},a]^n}-f(z)\log f(z)=\int_{z\in[0,a^{-1}]^n}\sum_{x\in S'_a}-f(x+z)\log f(x+z)
\end{equation}
we obtain that there is a fixed $\gamma\in [0,a^{-1}]^n$ such that
\begin{equation}\label{eqa2b}
a^{-n}\sum_{x\in S'_a}-f(x+\gamma)\log f(x+\gamma)
\end{equation}
is equal to the left hand side of (\ref{eqa2}). On the other hand, the left hand side of (\ref{eqa2}) is equal to $\mathbb D(X+\sigma M)+o(1)$. It remains to show that
\begin{equation}\label{eqa3}
a^{-n}\sum_{x\in S'_a}\bigl(f(\alpha(x))\log f(\alpha(x))-f(\beta(x))\log f(\beta(x))\bigr)=o(1)
\end{equation}
where $\beta(x)=x+\gamma$. (Note that $\alpha, \beta, \gamma$ all depend on $a$.) We will use that $\|\alpha(x)-\beta(x)\|_\infty\leq 2/a$ holds for every $x\in S'_a$ and thus $r(x):=\|\alpha(x)-\beta(x)\|_2\leq 2\sqrt n/a$. Let $t_x=f(\alpha(x))/f(\beta(x))$. We have by Lemma \ref{arany} that for every $\varepsilon>0$ if $a$ is big enough, then $t_x\geq 1-\varepsilon$ (resp. $t_x^{-1}\geq 1-\varepsilon$) provided that $f(\beta(x))> c$ (resp. $f(\alpha(x))>c$) where $c= \sigma^{-n/2}\exp(-a/(32n\sigma^2)).$ This implies that for every $a>0$ we can choose $\varepsilon=\varepsilon(a)$ such that $\lim_{a\rightarrow \infty}\varepsilon(a)=0$ and the previous property holds with $\varepsilon$. We will assume that $a$ is so large that $\varepsilon(a)<1/3$.

Let $T_1$ denote the sum of the terms in (\ref{eqa3}) where $f(\beta(x))\leq 2c$ and let $T_2$ denote the sum of the remaining terms. According to Lemma \ref{arany} either $f(\alpha(x))<c$ or $t_x^{-1}\geq 1-\varepsilon$. If $f(\beta(x))\leq 2c$, then we have  $f(\alpha(x))\leq 2c/(1-\varepsilon)\leq 3c$ in both cases. It follows that $T_1\leq 3c\log(3c)(2a^2+1)^n/a^n=o(1)$. 

Now we estimate $T_2$. From now on we assume that $f(\beta(x))>2c$.  By Lemma \ref{arany} we obtain $t_x\geq 1-\varepsilon$ holds and thus $f(\alpha(x))\geq (1-\varepsilon)2c>c$. This implies again by Lemma \ref{arany} that $t_x^{-1}\geq 1-\varepsilon$ and so $|1-t_x|\leq 2\varepsilon$. We have that $$f(\alpha(x))\log f(\alpha(x))-f(\beta(x))\log f(\beta(x))=f(\alpha(x)) \log t_x-(1-t_x)f(\beta(x))\log f(\beta(x)).$$
Using that $f(\alpha(x))=a^np_a(x)$ and $|\log t_x|\leq -\log(1-2\varepsilon)$ we get that
$$\biggl| a^{-n}\sum_{x\in S'_a, f(\beta(x))>2c} f(\alpha(x))\log t_x\biggr| \leq\sum_{x\in S_a}  p_a(x)(-\log(1-2\varepsilon)) =-\log(1-2\varepsilon)= o(1).$$
It is now clear that the following claim finishes the proof of the lemma.

\noindent{\it Claim:}~$a^{-n}\sum_{x\in S'_a} |f(\beta(x))\log f(\beta(x))| = O(1).$

By Lemma \ref{korlat} there is $b\in\mathbb{R}^+$ such that $f(x)<1$ whenever $\|x\|_\infty>b-1$. Let $B=[-b,b]^n$.
By the finiteness of $B$ we have that 
\begin{equation}\label{equa4}
a^{-n}\sum_{x\in S'_a\cap B} -\beta(x)\log(\beta(x))=\int_B -f\log f + o(1)
\end{equation}
and that
\begin{equation}\label{equa5}
a^{-n}\sum_{x\in S'_a\cap B} |\beta(x)\log(\beta(x))|=\int_B |f\log f| + o(1)
\end{equation}
It follows from (\ref{equa4}) and the property that the left hand side of (\ref{eqa2}) is equal to (\ref{eqa2b}) that  
\begin{equation}\label{equa6}
a^{-n}\sum_{x\in S'_a\setminus B} -\beta(x)\log(\beta(x))=\int_{\overline{B}} -f\log f+o(1).
\end{equation}
By (\ref{equa5}) and (\ref{equa6}) we get that
$$a^{-n}\sum_{x\in S'_a} |\beta(x)\log(\beta(x))|=\int_{\mathbb{R}^n} |f\log f|+o(1)=O(1).\qquad \qquad \square$$

\end{proof}	

\bigskip

In the following lemma, first we calculate the entropy of the random variable $t_{\sigma, a}(x)$ for every fixed $x$, and then we take the expectation of this quantity with respect to the distribution  of a random variable $X$. 

\begin{lemma}\label{varhato} Let $X$ be  a real valued random variable with finite variance and with distribution $\nu$. Then $\mathbb{E}_{\nu}(\mathbb{H}(t_{\sigma,a}(x)))=\log a+\mathbb D(N(0,\sigma))+o(1)$ for every fixed $\sigma>0$ as $a\rightarrow \infty$.
\end{lemma}

\begin{proof} Let $\nu_a$ denote the conditional distribution of $X$ when $|X|<a/2$ and let $\nu_a'$ denote the conditional distribution when $|X|\geq a/2$. We have that 
$$\mathbb{E}_{\nu}(\mathbb{H}(t_{\sigma,a}(x)))=\mathbb{P}(|X|<a/2)\mathbb{E}_{\nu_a}(\mathbb{H}(t_{\sigma,a}(x)))+\mathbb{P}(|X|\geq a/2)\mathbb{E}_{\nu_a'}(\mathbb{H}(t_{\sigma,a}(x))).$$
By Chebyshev's inequality we obtain that $\mathbb{P}(|X|\geq a/2)=O(a^{-2})$. It follows from the trivial uniform bound $\mathbb{H}(t_{\sigma,a}(x))\leq\log(a^2+1)$ that $\mathbb{P}(|X|\geq a/2)\mathbb{E}_{\nu_a'}(\mathbb{H}(t_{\sigma,a}(x)))=o(1)$. Similarly by $\mathbb{P}(|X|< a/2)=1-O(a^{-2})$ we obtain that $$\mathbb{P}(|X|<a/2)\mathbb{E_{\nu}}(\mathbb{H}(t_{\sigma,a}(x)))=\mathbb{E}_{\nu_{a}}(\mathbb{H}(t_{\sigma,a}(x)))+o(1).$$
It is now enough to prove that if $|x|\leq a/2$ then $\mathbb{H}(t_{\sigma,a}(x))=\log a +\mathbb D(N(0,\sigma))+o(1)$, where the $o(1)$ error term does not depend on $x$ but tends to $0$ as $a\rightarrow \infty$. 

Lemma \ref{finombecsles}  implies that $\mathbb H(t_{\sigma, a}(0))=\log a+\mathbb D(N(0, \sigma))+o(1)$; this is the $X=0$ case.
Next, suppose that $x\in S_a=\{r/a|r\in \mathbb Z, \| r\|_{\infty}\leq a^2\}$ and $0<x\leq a/2$. Notice that if $y\in S_a$ and $-a< y<  a-x$, then $\mathbb P(t_{\sigma, a}(0)=y)=\mathbb P(t_{\sigma, a}(x)=y+x)$, because the distance of $0$ and $x$ is a multiple of the distance of the points in the grid $S_a$ that we used for discretization. Hence the difference $\mathbb H(t_{\sigma, a}(x))-\mathbb H(t_{\sigma, a}(0))$ contains only terms corresponding to $|y-x|>a/2, y\in S_a$ in the first entropy expression (for $x$) and  $|y|>a/2, y\in S_a$ in the second one (for $0$).  The facts that $t_{\sigma, a}$ is supported on a set of at most $a^2+1$ elements and that the probability that a Gaussian distribution is farther from its expectation than $a/2$ is $O(\exp(-a^2)/2)$ imply that $\mathbb H(t_{\sigma, a}(x))-\mathbb H(t_{\sigma, a}(0))=o(1)$ uniformly in $0<x<a/2$ as $a\rightarrow \infty$, when $x\in S_a$. A similar argument works for $-a/2<x<0$ if $x$ is an element of $S_a$. 

Finally, let $x\in [-a/2, a/2]$ arbitrary, and $\overline x$ be the closest element of $S_a$ to $x$. As it is well known (e.g.\ as a consequence of Pinsker's inequality),  the total variation distance of $\textrm N(x, \sigma)$ and $\textrm N(\overline x, \sigma)$ is of order $O(1/a)$ provided $|x-\overline x|\leq 1/a$. By choosing an appropriate coupling of these two distributions and applying the same discretization, it follows that  $d_{\rm TV}(t_{\sigma,a}(x), t_{\sigma, a}(\overline x))=O(1/a)$. Applying Theorem 17.3.3. of \cite{cover} we obtain that  
\[|\mathbb H(t_{\sigma, a}(x))-\mathbb H(t_{\sigma, a}(\overline x))|\leq -d_{\rm TV}(t_{\sigma,a}(x), t_{\sigma, a}(\overline x))\log \frac{d_{\rm TV}(t_{\sigma,a}(x), t_{\sigma, a}(\overline x))}{a^2+1}=o(1), \]
and the error term does not depend on $x$. This concludes the proof. \hfill $\square$
\end{proof}

We finish this chapter with the proof of Theorem \ref{mainineq}.
Let $\mu_{\sigma,a}$ denote the process in which we pointwise discretize $\mu$ using $t_{\sigma,a}$ (using the notation of Chapter \ref{typpro}, we have $\mu_{\sigma, a}=t_{\sigma, a}*\mu$) and let $\mu_\sigma$ denote the process obtained from $\mu$ by adding $\sigma$ times the i.i.d normal distribution. 
By Lemma \ref{finombecsles} and $|B_k(C)|-(d/2)|B_k(e)|=1$ we obtain that 
$$\mathbb{H}(B_k(C),\mu_{\sigma,a})-\frac{d}{2}\mathbb{H}(B_k(e),\mu_{\sigma,a})=\log a + \mathbb D(B_k(C),\mu_\sigma)-\frac{d}{2}\mathbb D(B_k(e),\mu_\sigma)+o(1).$$
By Theorem \ref{discmain} and the typicality of $\mu$ we get that
$$\mathbb{H}(B_k(C),\mu_{\sigma,a})-\frac{d}{2}\mathbb{H}(B_k(e),\mu_{\sigma,a})\geq \mathbb{E}_{\mu_o}(\mathbb{H}(t_{\sigma,a}(x))).$$
Using the previous formulas, Lemma \ref{varhato} and the limit $a\to\infty$ we obtain that
\begin{equation}\label{eq:eht}\mathbb  D(B_k(C),\mu_\sigma)-\frac{d}{2} \mathbb D(B_k(e),\mu_\sigma)\geq D(N(0,\sigma)\end{equation} for every $\sigma>0$.

Let $S\subset V_d$ be either $B_k(C)$ or $B_k(e)$. We denote by $\mu_{S, \sigma}$ the probability measure obtained by convolving the measure $\mu_S$ with the standard normal distribution on $W_{\lambda}(S)$ (for the definition of $W_\lambda(S)$ see chapter \ref{chap:eigen}), where $\lambda$ is the eigenvalue corresponding to $\mu$. Observe that the standard normal distribution on $\mathbb R^S$ is the independent sum of the standard normal distribution on $W_{\lambda}(S)$ and on $W_{\lambda}(S)^{\perp}$. Then by using that $\mathrm{dim}\, W_{\lambda}(S)^{\perp}=|S|-|\partial S|$, we have 
\[\mathbb D(S, \mu_{\sigma})=\mathbb D_{\rm sp}(\mu_{S, \sigma})+(|S|-|\partial S|)\mathbb D(N(0,\sigma)).\]
 Using this formula for $S=B_k(C)$ and $S=B_k(e)$ in \eqref{eq:eht} together with $|\partial B_k(C)|=(d/2)|\partial B_k(e)|$, we obtain that 
 \[\mathbb D_{\rm sp}(\mu_{B_k(C), \sigma})-\frac{d}{2}\mathbb D_{\rm sp}(\mu_{B_k(e), \sigma})\geq 0.\]
 If $\sigma\rightarrow 0$, then we obtain the statement of Theorem \ref{mainineq}.
   \hfill $\square$

\section{Eigenvalues of the covariance operator}

\label{chap:cov}

The main goal of this section is to prove the following statement on the differential entropies of Gaussian waves. Recall the definition of $\mathbb D_{\rm sp}$ from Chapter \ref{smooth}.  In addition, by $\mathrm{det}_{\rm sp}$ we mean the product of the non-zero eigenvalues of a matrix.

\begin{theorem} \label{thm:nulla} Let $\Psi_{\lambda}$ be a Gaussian wave function on the $d$-regular tree with  $\lambda\in [-2\sqrt{d-1}, 2\sqrt{d-1}]$. Then we have 
\[\mathbb D_{\rm sp}(B_k(C), \Psi_{\lambda})-\frac d2 \mathbb D_{\rm sp}(B_k(e), \Psi_{\lambda})\rightarrow 0 \qquad (k\rightarrow \infty).\]
\end{theorem}

Let $\Sigma_k$ be the covariance matrix of the joint distribution of $\Psi_{\lambda}$ restricted to the ball $B_k(C)$, and $\Sigma'_k$ be the similar covariance matrix on $B_k(e)$. The differential entropy of a multivariate normal random variable with covariance matrix $\Sigma$ of rank $m$ is given by $\frac 12 \log \big((2\pi e)^m \det_{\rm sp} \Sigma)\big)$ if we measure differential entropy inside the support of the variable. Equation \eqref{eq:dim} and Corollary \ref{cor:qsp} imply that the rank of $\Sigma_k$ is $d/2$ times the rank of $\Sigma_k'$. Hence  we need  to prove that 
\begin{equation}\label{eq:logdet}\log \det\!_{\rm sp}\, \Sigma_k-\frac{d}{2}\log \det\!_{\rm sp} \, \Sigma_k'\rightarrow 0 \qquad (k\rightarrow\infty).\end{equation}

Notice that if $s$ is an eigenvalue of both $\Sigma_k$ and $\Sigma_k'$, and its multiplicity in the first case is $d/2$ times its multiplicity in the second case, then it is canceled out in the difference. In order to find the eigenvalues that do not cancel out, we decompose both $\mathbb R^{|B_k(C)|}$ and $\mathbb R^{|B_k(e)|}$ as a union of orthogonal subspaces that are invariant under the corresponding covariance operators. 

First we need some notation. We will use the genealogical labeling of the vertices in $B_k(C)$ and in $B_k(e)$ (in this section, we will not distinguish vertices and labels). 
In $B_k(C)$, the root gets label $\emptyset$, and we put the labels on the vertices such that the labels of neighbors differ only in the last coordinate (i.e.\ $1, 2, \ldots, d$ are the neighbors of the root; $11, 12, \ldots, 1(d-1)$ are the further neighbors of $1$, and so on).  For a vertex $v$ the length of its label is denoted by $|v|$, which is its distance  from $\emptyset$. For  a vertex $v$ and a sequence $y$, by $vy$ we mean the vertex with the label obtained by concatenating $v$ and $y$. We say that $v$ is an ancestor of $w$ (denoted by $v\rightarrow w$), if $w=vy$ for some $y\neq \emptyset$. As for $B_k(e)$, we use a similar notation, but keeping track of symmetry with respect to the central edge $e$. The endpoints of $e$ have labels $\emptyset$ and $\emptyset'$. The descendants of $\emptyset$ have labels $1, \ldots, d-1$, their descendants have labels $11, 12, \ldots, 1(d-1), 21, \ldots$ and so on. Similarly, the descendants of $\emptyset'$ have labels $1', \ldots, (d-1)'$, their descendants have labels $11', 12', \ldots, 1(d-1)', 21', \ldots$ and so on.  Note that $|v|$ still denotes the length of the label. 

We assign a linear subspace to each vertex in $B_k(C)\setminus \partial B_k(C)$ and $B_k(e)\setminus \partial B_k(e)$. Fix $v\in B_k(C)\setminus \partial B_k(C)$. Let $\mathcal E_v$ be the elements $\alpha\in\mathbb R^{B_k(C)}$ for which the following hold. $(i)$ $\alpha_w=0$ if $v$ is not an ancestor of $w$. $(ii)$ Suppose that $1\leq j\leq (d-1)$ and $y, z$ are labels with $|y|=|z|$. Then $\alpha_{vjy}=\alpha_{vjz}$. (To put it in another way, for descendants of $v$, the value of $\alpha$ depends only on the first coordinate after $v$ and the distance from $v$.) $(iii)$ We have $\sum_{y: |y|=r} \alpha_{vy}=0$ for $r\geq 1$. 

In addition, we introduce the following subspace: 
\[\mathcal G=\{\alpha\in\mathbb R^{B_k(C)}: \alpha_v=\alpha_w \text{ if } |v|=|w|\}.\] 
We will also refer to $\mathcal E_v'$ (when $v\in B_k(e)\setminus \partial B_k(e)$), which are linear subspaces of $\mathbb R^{B_k(e)}$ defined similarly. 
The definition of the complement subspace is somewhat different:
\[\mathcal G_1'=\{\alpha\in\mathbb R^{B_k(e)}: \alpha_v=\alpha_w \text{ if } |v|=|w|\}.\]
\[\mathcal G_2'=\{\alpha\in\mathbb R^{B_k(e)}: \alpha_v=\alpha_w \text{ if } \emptyset \rightarrow v, \emptyset \rightarrow w, |v|=|w|;\ \alpha_{v'}=-\alpha_v \text{ if } v=\emptyset \text{ or } \emptyset \rightarrow v\}.\]

\begin{lemma}\label{lem:ev}
The following  hold for the linear subspaces defined above. 
\begin{enumerate}[$(a)$]
\item $\mathcal E_v, v\in B_k(C)\setminus \partial B_k(C)$ and $\mathcal G$ are invariant under $\Sigma_k$. Similarly, $\mathcal E_v', v\in B_k(e)\setminus \partial B_k(e)$ and $\mathcal G_1', \mathcal G_2'$ are invariant under $\Sigma_k'$. 
\item $\mathcal E_v, v\in B_k(C)\setminus \partial B_k(C)$ and $\mathcal G$ are pairwise orthogonal. Similarly, $\mathcal E_v', v\in B_k(e)\setminus \partial B_k(e)$ and $\mathcal G_1', \mathcal G_2'$ are pairwise orthogonal. 
\item $\mathbb R^{B_k(C)}=\mathcal G + \sum_{v\in B_k(C)} \mathcal E_v$ and $\mathbb R^{B_k(e)}=\mathcal G_1'+ \mathcal G_2'+ \sum_{v\in B_k(e)} \mathcal E'_v$.
\end{enumerate}
\end{lemma}

\begin{proof} Before going into the proof, we note that we will only use the property that every entry $(i,j)$ of $\Sigma_k$ depends only on the distance of $i$ and $j$. $(a)$ First observe $\mathcal G$ consists of all vectors that are invariant under the full automorphism group of $B_k(C)$. This property is preserved by $\Sigma_k$ and thus $\mathcal G$ is invariant under $\Sigma_k$. 

Take any $\alpha\in \mathcal E_v$. For $\Sigma_k \alpha$, property $(ii)$ is preserved because the entries of $\Sigma_k$ depend only on the distance of the two corresponding vertices (as it is the covariance matrix of an invariant random process). Putting this together with the third property we get that $(i)$ also holds for $\Sigma_k \alpha$. Property $(iii)$ means orthogonality to $\mathcal G$; using the invariance of $\mathcal G$, this will also be satisfied by $\Sigma_k\alpha$, which is thus in $\mathcal E_v$. Similar arguments work for the other two linear subspaces.  
\end{proof} 

$(b)$ Fix $v_1, v_2\in B_k(C)$. If none of them is ancestor of the other one, then the support of any vector of $\mathcal E_{v_1}$ is disjoint from the support of any vector in $\mathcal E_{v_2}$, which implies orthogonality. If $v_1\rightarrow v_2$, then the value of a vector in $\mathcal E_{v_1}$ is the same at all vertices of type $v_2y$ with $|y|$ fixed. Multiplying this by the values of a vector in $\mathcal E_{v_2}$ and summing this up for different $y$s (of fixed length) we get $0$, because of property $(iii)$. This implies the orthogonality. The other cases are similar; we omit the details. 

$(c)$ The dimensions of these subspaces are as follows. 
\[\mathrm{dim}\ \mathcal E_{v}=(k-|v|+1)(d-2) \qquad \text{for } \emptyset \neq v\in B_k(C)\setminus \partial B_k(C);\]
\[\mathrm{dim}\ \mathcal E_{\emptyset}=(k+1)(d-1); \qquad \dim\  \mathcal G=k+2;\] 
\[\mathrm{dim}\ \mathcal E'_{v}=(k-|v|+1)(d-2) \qquad \text{for } \emptyset \neq v\in B_k(e)\setminus \partial B_k(e);\]
\[\mathrm{dim}\ \mathcal G'_{1}=k+1; \qquad \dim\  \mathcal G_2'=k+1.\]
The following equalities are easy to check by induction on $k$: 
\[1+(k+1)d+\sum_{j=1}^{k} d(d-1)^{j-1} (k+1-j)(d-2)=|B_k(C)|;\]
\[2(k+1)+2\sum_{j=1}^{k-1} (d-1)^{j-1} (k-j)(d-2)=|B_k(e)|. \]
Hence the sum of the dimension of the linear subspaces $\mathcal E_v$ and $\mathcal G$ is equal to the dimension of the space $\mathbb R^{B_k(C)}$. Since the subspaces are pairwise orthogonal by part $(b)$ of the lemma, this implies that the sum must be equal to $\mathbb R^{B_k(C)}$. A similar argument works for $B_k(e)$.
\hfill $\square$

For the following lemma, recall the definition of $f(k, x)$ from equation \eqref{eq:fkx}. Furthermore, the calculation about this recurrence relation in \cite{alon} imply that if we take $\lambda=2\sqrt{d-1}x$, then the covariance of the values at distance $k$ in the Gaussian wave $\Psi_{\lambda}$ is equal to $f(k,x)$. 

\begin{lemma}\label{lem:s1}
Let $f(k, x)$ be defined by equation \eqref{eq:fkx}. We define
\[l(k, x)=1+\sum_{j=1}^{k-1}(d-2)(d-1)^{j-1} f(2j, x).\]
Then the eigenvalues of $\Sigma_k$ corresponding to $\mathcal E_{\emptyset}$ and $\mathcal G$ are as follows. 
\[s_1(k, x)=1+\sum_{j=1}^k (d-1)^j f(2j, x)+ \sum_{j=1}^k l(j,x)\quad \text{ with multiplicity }\mathrm{1};\]
\[s_2(k, x)=\bigg(1-s_1(k, x)+\sum_{j=1}^k dl(j,x)\bigg)/(d-1) \quad \text{ with multiplicity }\mathrm{d-1}.\]
The eigenvalues of $\Sigma_k'$ corresponding to $\mathcal G_{1}'$ and $\mathcal G_2'$ are as follows. 
\[s_3(k, x)=\sum_{j=1}^k l(j, x)+(d-1)^{j-1}f(2j-1, x) \quad \text{ with multiplicity }\mathrm{1};\]
\[s_4(k, x)=\sum_{j=1}^k l(j, x)-(d-1)^{j-1}f(2j-1, x)\quad \text{ with multiplicity }\mathrm{1}.\]
\end{lemma}
\begin{proof} Using the notation from Chapter \ref{chap:eigen}, let $\mathcal S_k=W_{\lambda}(B_k(C))$ and $\mathcal S'_k=W_{\lambda}(B_k(e))$.
First we consider $\mathcal G$. Since $f(k,x)$ is the covariance of the values at distance $k$ in a (nontrivial) Gaussian wave, we have by linearity that every row of $\Sigma_k$ is in $S_k.$ It follows that $\mathrm{Im}(\Sigma_k)\subseteq \mathcal S_k$. On the other hand, by the previous lemma, $\mathcal G$ is invariant under $\Sigma_k$. Notice that $\mathcal G\cap \mathcal S_k$ is one dimensional: given the value at the root, the common value of its neighbors is determined (even for $\lambda=0$), and this can be continued. It follows that $\Sigma_k|_\mathcal G$ has rank one, and the eigenvalue corresponding to $\mathcal G$ can be obtained by  calculating the trace of the matrix of $\Sigma_k|_\mathcal G$ in an arbitrary basis. By choosing the basis of the indicator functions of the spheres of radius $0, 1, \ldots, k$ around the root, elementary calculation shows that this eigenvalue is equal to $s_1(k, x)$.

The three eigenvalues corresponding to the invariant subspaces $\mathcal E_{\emptyset}$, $\mathcal G_1'$ and $\mathcal G_2'$ can be obtained by similar arguments, by identifying the image of $\Sigma_k$ (or $\Sigma_k'$) restricted to the given subspace and calculating the trace of its matrix. In the second case, $\mathcal E_{\emptyset}'\cap \mathcal S_k$ has dimension $d-1$: 0 is assigned to the root; the values of the $d$ neighbors of $\emptyset$ have to sum up to 0, but there are no other conditions; given these values, all the others are uniquely determined. The only nonzero eigenvalue has multiplicity $d-1$, which makes it possible to calculate it based on the trace of the matrix. As for the last two cases, $\mathcal G_1'\cap \mathcal S_k'$ and $\mathcal G_2'\cap \mathcal S_k'$ both have dimension 1 again: given the value at $\emptyset$, the value at $\emptyset'$ has to be the same or the opposite. Then the eigenvalue equation and the equality conditions in $\mathcal G_1'$ and $\mathcal G_2'$ uniquely determine all the other values. By choosing appropriate bases in these subspaces, it is straightforward to obtain the eigenvalues in the lemma. \end{proof}\hfill $\square$

\begin{lemma}\label{lem:kul}
Using the notation of the previous lemma, for every $x\in [-1, 1]$, we have 
\[\log s_1(k, x)+(d-1)\log s_2(k, x)-\frac d2 \log s_3(k, x)-\frac d2 \log s_4(k, x)\rightarrow 0 \qquad (k\rightarrow \infty).\]
\end{lemma}
\begin{proof}
First we calculate the middle term of $s_1(k, x)$. Using equation \eqref{eq:fkx}, we obtain that
\[T(k, x):=\sum_{j=1}^k (d-1)^j f(2j, x)=\sum_{j=1}^k \frac{(d-1)^j}{\sqrt{d(d-1)^{2j-1}}}q_{2j}(x),\]
where the polynomials $q$ are defined by equation \eqref{eq:qkx}.
Straightforward calculation shows that with $x=\cos \vartheta$  we have
\[T(k, x)=\frac{d-1}{d}U_{2k}(x)+\frac{d-2}{d\sin \vartheta} \mathrm{Im}\frac{e^{3i\vartheta}(e^{(2k-2)i\vartheta}-1)}{e^{2i\vartheta}-1}-\frac 1d,\]
if $\sin \vartheta\neq 0$ and $e^{2\vartheta}\neq 1$. (We deal with the exceptional cases at the end of the proof.)
Using this formula, we obtain that 
\begin{equation*}l(k, x)=1-\frac{d-2}{(d-1)d}+\frac{d-2}{d}\,U_{2k-2}(x)+\frac{(d-2)^2}{d(d-1)\sin\vartheta}\,\mathrm{Im}\ \frac{e^{3i\vartheta}(e^{(2k-4)i\vartheta}-1)}{e^{2i\vartheta}-1} \qquad (k\geq 2).\end{equation*}
Notice that the last term is bounded in $k$ for every fixed $x$. From now on, $O(1)$ will denote a quantity which depends both on $x$ and $k$ such that {\em for every fixed $x$ it is bounded in $k$}. We emphasize that in Theorem \ref{thm:nulla} the limit is taken for fixed $\lambda$ (which is equal to $x\cdot \sqrt{d-1}$). Thus in the proofs of this chapter we always think of $x$ as a fixed quantity while tending to infinity with $k$. 

Continuing our calculations, we obtain that
\[s_1(k, x)=1+T(k, x)+\sum_{j=1}^k l(j, x)=\bigg(1-\frac{d-2}{d(d-1)}-\frac{(d-2)^2}{d(d-1)\sin \vartheta}\mathrm{Im}\, \frac{e^{3i\vartheta}}{e^{2i\vartheta}-1}\bigg)k+O(1).\]
On the other hand, we have 
\[s_1(k, x)-s_2(k, x)=\frac{d}{d-1}s_1(k,x)-\frac{d}{d-1}\sum_{j=1}^k l(j, x)-\frac{1}{d-1}=O(1);\]
\[s_1(k, x)-s_3(k, x)=1+\sum_{j=1}^k (d-1)^j f(2j, x)-(d-1)^j f(2j-1, x)=O(1);\]
\[s_3(k, x)-s_4(k, x)=2\sum_{j=1}^k (d-1)^j f(2j-1, x)=O(1).\]
To put it in another way, with 
\[A=1-\frac{d-2}{d(d-1)}-\frac{(d-2)^2}{d(d-1)\sin \vartheta}\,\mathrm{Im}\, \frac{e^{3i\vartheta}}{e^{2i\vartheta}-1}\]
the expressions $s_1(k, x)$, $s_2(k, x)$, $s_3(k, x)$ and $s_4(k, x)$ are all in the form $Ak+O(1)$. If $A>0$, then we get that \[\log s_j(k, x)=\log A + \log k+o(1)  \qquad (j=1, 2, 3, 4),\] which implies the statement of the lemma. Hence, in the rest of the proof, we check that $A>0$ holds. 

First notice that $1-\frac{d-2}{d(d-1)}>\frac{(d-2)^2}{d(d-1)}$ is satisfied for all $d$. In addition, by using elementary trigonometric identities we obtain
\[\mathrm{Im}\,\frac{e^{3i\vartheta}}{\sin \vartheta(e^{2i\vartheta}-1)}=\frac{\sin \vartheta-\sin 3\vartheta}{2\sin\vartheta(1-\cos 2\vartheta)}=\frac{1-\cos 2\vartheta-2\cos^2\vartheta}{4 \sin^2\vartheta}=\frac{\sin^2\vartheta-\cos^2\vartheta}{2\sin^2 \vartheta}\leq 1.\]
Putting this together, we get the positivity of $A$, which concludes the proof.

We have to deal with the remaining special cases.  First, if $|x|=1$, then $\sin \vartheta=0$, but we still have $U_k(\cos \vartheta)=k+1$. This implies that $l(k, x)$ is a quadratic polynomial, and $s_1(k, x)$ is a cubic polynomial (with leading coefficient $1-(d-2)^2/d/(d-1)>0$). Moreover, the differences $s_1-s_2, s_1-s_3, s_1-s_4$ are all of order $O(k^2)$, which implies the statement of the lemma.

The last case is when $e^{2i\vartheta}=1$. This implies $U_k(\cos \vartheta)=1$ for all $k$.  That is, $q_k(x)=(d-2)/\sqrt{d(d-1)}$, and $l(k, x)$ is of the form $Bk+O(1)$ for some nonzero $B$. Now $s_1, s_2, s_3, s_4$ are all quadratic polynomials as a function of $k$, while their differences are linear. It follows again that the expression in the lemma goes to 0 as $k\rightarrow\infty$. \hfill $\square$
\end{proof}

\noindent{\it Proof of Theorem \ref{thm:nulla}.} As we have discussed, it is sufficient to show that \eqref{eq:logdet} holds. First fix $k$, and recall Lemma \ref{lem:ev}. Notice that for every $1\leq r\leq k-1$, if we take two vertices $v_1, v_2$ in $B_k(C)$ such that $|v_1|=|v_2|=r$, then the linear transformation $\Sigma_k$ restricted to $\mathcal E_{v_1}$ is isomorphic to the linear transformation $\Sigma_k$ restricted to $\mathcal E_{v_2}$ (we use again that the entries of the covariance matrix depend only on the distance of the vertices). Hence the set of eigenvalues of $\Sigma_k$ corresponding to the invariant subspaces $\mathcal E_{v_1}$ and $\mathcal E_{v_2}$ are the same. Furthermore, this linear transformation is 
isomorphic to the linear transformation $\Sigma_k'$ restricted to $\mathcal E_{v}'$, if $|v|=r$ holds.  For every $1\leq r\leq k-1$ we have 
\[\big|\{v: |v|=r,\, v\in B_k(C)\}\big|=\frac d2 \big|\{v: |v|=r,\, v\in B_k(e)\}\big|.\] 
This means that all eigenvalues belonging to the spaces $\mathcal E_v$ and $\mathcal E'_v$ cancel out  in the expression in \eqref{eq:logdet} for $|v|>0$. 

Therefore only the eigenvalues corresponding to $\mathcal G, \mathcal E_{\emptyset}', \mathcal G_1'$ and $\mathcal G_2'$ are left. These are calculated with multiplicities in Lemma \ref{lem:s1}, and hence Lemma \ref{lem:kul} finishes the proof by showing that the difference goes to zero as $k\rightarrow \infty$. \hfill $\square$

\section{Improved differential entropy inequality}
\label{chap:impr}

In this chapter we use a combination of Theorem \ref{mainineq} and Theorem \ref{thm:nulla} to prove an improved version of Theorem \ref{mainineq} for the case $k=0$. We will use the notation from Chapter \ref{chap:eigen}. If $S$ is either $B_k(C)$ or $B_k(e)$ for some $k\geq 0$ and $p$ is in $S\setminus\partial S$, then we have for the Gaussian wave $\Psi_\lambda$ with distribution $\mu$ that

\begin{equation}\label{dfor}
\mathbb{D}(\mathcal{Q}(S,p,\mu))=\mathbb{D}(B_p)+\sum_{(v,w)\in D}\mathbb{D}(A_{v,w}).
\end{equation}
 We obtain from (\ref{dfor}) that  the differential entropy $\mathbb{D}(\mathcal{Q}(S,p,\mu))$ does not depend on $p$ for a Gaussian wave with distribution $\mu$. Observe that changing the vertex $p$ to $p'$ results in a linear transformation $T_{p,p'}$ in the system $\mathcal{Q}(S,p,\mu)$. The invariance of the differential entropy shows by Lemma \ref{difftrdet} that $T_{p,p'}$ has determinant $1$. 
 
Now applying Lemma \ref{difftrdet} for an arbitrary smooth eigenvector process $\nu$ with eigenvalue $\lambda$ we obtain that the value of $\mathcal{D}(\mathcal{Q}(S,p,\nu))$ is independent of $p$ since the transformations $T_{p,p'}$ depend only on the quadruple $p,p',\lambda,S$ and can be calculated from the eigenvector equation. For a general smooth eigenvector process $\nu$ we define $\mathbb{D}(S,\nu)$ as this unique differential entropy of $\mathbb{D}(\mathcal{Q}(S,p,\nu))$. We will need the next definition.
 
 \begin{definition} A process $\{Y_v\}_{v\in V_d}$ in $I_d(X)$ is called $2$-Markov if for an arbitrary edge $e$ the distributions $\{Y_v\}_{v\in W_1}$ and $\{Y_v\}_{v\in W_2}$ are conditionally independent with respect to $\{Y_v\}_{v\in e}$ where $W_1$ and $W_2$ are the set of vertices on the two sides of $e$. (With this notation $V_d=W_1\cup e\cup W_2$.)
 \end{definition}
 
 Note that the $2$-Markov property implies that the marginal distribution $Y_C=\{Y_v\}_{v\in C}$ determines the whole process because we can build up the distribution $\{Y_v\}_{v\in V_d}$ using iterated conditionally independent couplings of $Y_C$ along edges. More precisely, if for some connected subgraph $K$ of $T_d$ the distribution $\{Y_v\}_{v\in V(K)}$ is already constructed and $w\in V_d$ is a vertex such that the star $B_1(w)$ intersects $K$ in a single edge $e$,  then the joint distribution $\{Y_v\}_{v\in V(K)\cup B_1(w)}$ is the conditionally independent coupling of $\{Y_v\}_{v\in V(K)}$ and $\{Y_v\}_{v\in B_1(w)}$ with respect to $\{Y_v\}_{v\in e}$. The invariance of the process implies that $\{Y_v\}_{v\in B_1(w)}$ has the same distribution as $Y_C$. By iterating this procedure we can build up the marginal distribution on any finite connected subgraph of $T_d$ and thus the whole process in uniquely determined.
 
 \begin{lemma}\label{imprlem2} We have for every $k\geq 1$ and smooth eigenvector process $\nu$ the following three inequalities. $$\mathbb{D}(B_k(C),\nu)\leq d\mathbb{D}(B_k(e),\nu)-(d-1)\mathbb{D}(B_{k-1}(C),\nu),$$
 $$\mathbb{D}(B_k(e),\nu)\leq 2\mathbb{D}(B_{k-1}(C),\nu)-\mathbb{D}(B_{k-1}(e),\nu),$$
 $$\mathbb{D}(B_k(C),\nu)-(d/2)\mathbb{D}(B_k(e),\nu)\leq\mathbb{D}(B_{k-1}(C),\nu)-(d/2)\mathbb{D}(B_{k-1}(e),\nu).$$ If $\nu$ is Gaussian, then we have equality everywhere. Furthermore if we have equality everywhere (for every $k$), then $\nu$ is $2$-Markov.
 \end{lemma}
 
 \begin{proof} The proof is based on the general fact (see Lemma \ref{diffincform}) that for a joint distribution $(X,Y,Z)$ we have that $\mathbb{D}(X,Z)+\mathbb{D}(Y,Z)-\mathbb{D}(Z)\geq\mathbb{D}(X,Y,Z)$ holds with equality if and only if $X$ and $Y$ are conditionally independent with respect to $Z$. To see the first inequality let us place $p$ to the root of $B_k(C)$. We can cover $B_k(C)$ in a rotational symmetric way by $d$ copies of $B_k(e)$ in a way that all of them contain $B_{k-1}(C)$ and they are disjoint outside of $B_{k-1}(C)$. Each copy of $B_k(e)$ covers a subset of the variables $A_{v,w}$ and $B_p$ such that the joint differential entropy of this subset of variables is equal to $\mathbb{D}(B_k(e),\nu)$. Now Lemma \ref{diffincform} finishes the proof of the first inequality. The other two inequalities can be seen in a similar way. If we have equality everywhere for every $k$, then by lemma \ref{diffincform} we get that the joint distribution of $\nu$ on $B_k(C)$ is the conditionally independent coupling of $d$ copies of $B_k(e)$ over $B_{k-1}(C)$ and the joint distribution on $B_k(e)$ is the conditionally independent coupling of $2$ copies of $B_ {k-1}(C)$ over $B_{k-1}(e)$. By induction the $2$-markov property follows inside $B_k(C)$ for every $k$ and thus for the whole process. 
 \end{proof}

 Let $\alpha(S,\nu)$ denote the difference $\mathbb{D}(\mathcal{Q}(S,p,\nu))-\mathbb{D}(\mathcal{Q}(S,p,\mu))$ where $\mu$ is the Gaussian eigenvector process with the same eigenvalue as $\nu$. If we apply the same change of basis to both $\mathbb{D}(\mathcal{Q}(S,p,\nu))$ and $\mathbb{D}(\mathcal{Q}(S,p,\mu))$, they change with the same additive constant by lemma \ref{difftrdet} and thus $\alpha(S,\nu)$ remains unchanged. This shows the basis independence of $\alpha(S,\nu)$. In particular we have that if $\nu$ is a smooth eigenvector process then $\alpha(S,\nu)=\mathbb{D}_{\rm sp}(S,\nu)-\mathbb{D}_{\rm sp}(S,\mu)$ and so
$$\alpha(B_k(C),\nu)-(d/2)\alpha(B_k(e),\nu)=$$ 
\begin{equation}\label{imprlem3eq}
 \Bigl(\mathbb{D}_{\rm sp}(B_k(C),\nu)-(d/2)\mathbb{D}_{\rm sp}(B_k(e),\nu)\Bigr)-\Bigl(\mathbb{D}_{\rm sp}(B_k(C),\mu)-(d/2)\mathbb{D}_{\rm sp}(B_k(e),\mu)\Bigr).
 \end{equation}

Using Theorem \ref{mainineq} and Theorem \ref{thm:nulla} we get the following consequence. 

\begin{proposition}\label{improp} If $\nu$ is a smooth typical eigenvector process  then $$\limsup_{k\to\infty} \alpha(B_k(C),\nu)-(d/2)\alpha(B_k(e),\nu)=\limsup_{k\to\infty}\mathbb{D}_{\rm sp}(B_k(C),\nu)-(d/2)\mathbb{D}_{\rm sp}(B_k(e),\nu)\geq 0.$$
 \end{proposition}
 
 \begin{proof}  By Theorem \ref{mainineq} the first term in (\ref{imprlem3eq}) is non negative and by Theorem \ref{thm:nulla} the second term converges to $0$. This completes the proof.
 \end{proof}
 
 The main theorem of this chapter is the following.
 
 \begin{theorem}\label{thm:imprmain} If $\nu$ is a smooth typical eigenvector process, then $\alpha(C,\nu)-(d/2)\alpha(e,\nu)\geq 0$ and equality implies that $\nu$ is $2$-Markov.
 \end{theorem}
 
 \begin{proof} By iterating the third inequality in Lemma \ref{imprlem2} we obtain that 
 $$\mathbb{D}(B_k(C),\nu)-(d/2)\mathbb{D}(B_k(e),\nu)\leq\mathbb{D}(B_l(C),\nu)-(d/2)\mathbb{D}(B_l(e),\nu)$$
 holds for every $k\geq l\geq 0$. Furthermore if we replace $\nu$ in the above formula by the Gaussian wave $\mu$ then we get equality. This implies that
\begin{equation}\label{imprtheq}
\alpha(B_k(C),\nu)-(d/2)\alpha(B_k(e),\nu)\leq\alpha(B_l(C),\nu)-(d/2)\alpha(B_l(e),\nu)
\end{equation}
 holds for every $k\geq l\geq 0$. By Proposition \ref{improp} we get the inequality of the theorem by applying \eqref{imprtheq} with $l=0$ and $k\to\infty$.
 For the second statement assume that $\alpha(C,\nu)-(d/2)\alpha(e,\nu)=0$. By (\ref{imprtheq}) this is only possible if $\alpha(B_k(C),\nu)-(d/2)\alpha(B_k(e),\nu)=0$ holds for every $k$ and thus the inequalities of Lemma \ref{imprlem2} are all equalities. This implies that $\nu$ is $2$-Markov. \hfill $\square$
 \end{proof}

 \section{Heat equation and the proof of the main theorem}

\label{chap:heat}

\begin{definition}\label{stardist}  For $\lambda\in [-2\sqrt{d-1},2\sqrt{d-1}]$ and $d\geq 3$ let $\mathcal{F}_{d,\lambda}$ denote the set of joint distributions $F=(X_1,X_2,\dots,X_d,Z)$ of real valued random variables such that
\begin{enumerate}
\item $\mathbb{E}(X_iX_j)=(\lambda^2-d)d^{-1}(d-1)^{-1}$, $\mathbb{E}(X_iZ)=\lambda/d$, $\mathbb E(X_i)=\mathbb E(Z)=0$  and $\mathbb{E}(X_i^2)=\mathbb{E}(Z^2)=1$ holds for $1\leq i,j\leq d$ and $i\neq j$;
\item the joint distribution $(X_1,X_2,\dots,X_d,Z)$ is symmetric under every permutation that fixes $Z$;
\item for every $1\leq i\leq d$ the joint distribution $(X_i,Z)$ is the same as the joint distribution $(Z,X_i)$;
\item the quantities $\mathbb D_{\rm sp}(X_1,X_2,\dots,X_d,Z)$ and $\mathbb D(X_1,Z)$ are both finite.
\end{enumerate}
We define the function $\mathcal{D}:\mathcal{F}_{d,\lambda}\rightarrow\mathbb{R}$ by $\mathcal{D}(F)=\mathbb D_{\rm sp}(X_1,X_2,\dots,X_d,Z)-\frac{d}{2}\mathbb D(X_1,Z)$.
\end{definition}

Notice that the covariance conditions of definition \ref{stardist} guarantee that $\mathbb{E}((X_1+X_2+\dots+X_d-\lambda Z)^2)=0$ and thus $X_1+X_2+\dots+X_d=\lambda Z$ holds with probability $1$.  This implies that the joint distribution $F$ is concentrated on the $1$ co-dimensional ($d$-dimensional) subspace $W_\lambda(C)$ in $\mathbb{R}^{d+1}$. The subspace differential entropy in definition \ref{stardist} is measured in this subspace.

In this chapter we think of $\lambda\in[-2\sqrt{d-1},2\sqrt{d-1}]$ and $d$ as fixed values and most of the times our notation will not indicate the dependence on these values even if there is such a dependence.  Our goal is to solve the extremal problem of maximizing $\mathcal{D}$ inside $\mathcal{F}_{d,\lambda}$ (see Theorem \ref{gausscond}). This will provide the last step in the proof of our main theorems (see Theorem \ref{main} and Theorem \ref{mainlim}) as we will explain at the end of this chapter.

 It will be important that there is a unique element $F^*=(X_1^*,X_2^*,\dots,X_d^*,Z^*)$ in $\mathcal{F}_{d,\lambda}$ such that $F^*$ is Gaussian. (The covariances define the Gaussian system uniquely, which clearly satisfies the symmetry conditions.) Note that $F^*$ depends on both $d$ and $\lambda$.
Most of this chapter deals with the proof of the the next theorem which says that the entropy formula $\mathcal{D}$ in $\mathcal{F}_{d,\lambda}$ is maximized by the Gaussian distribution $F^*$.

\begin{theorem}\label{gausscond} For every $F\in\mathcal{F}_{d,\lambda}$ we have that $\mathcal{D}(F)\leq\mathcal{D}(F^*)$ and equality holds if and only if $F=F^*$.
\end{theorem}

To get rid of the subspace differential entropy, we apply a change of basis to the systems in $\mathcal{F}_{d,\lambda}$. We can choose a fix linear transformation $T:\mathbb{R}^{d+1}\to \mathbb{R}^d$ (depending on $d$ and $\lambda$) such that for every $F\in\mathcal{F}_{d,\lambda}$ the system $T(F)=(B_1,B_2,\dots,B_d)$ satisfies $\mathbb{E}(B_iB_j)=\delta_{i,j}$ for $1\leq i,j\leq d$.  Using that $M=T(F^*)$ is Gaussian we obtain that $M$ is the standard normal distribution on $\mathbb{R}^d$. 
Using that linear transformations change differential entropy with a fix constant (depending on the transformation; see Lemma \ref{difftrdet}), the statement of the theorem is equivalent to the fact that 
$$\mathbb D(T(F))-(d/2)\mathbb D(X_1,Z)\leq \mathbb D(M)-(d/2) \mathbb D(X_1^*,Z^*)$$ and equality holds if and only if $T(F)=M$. The proof of Theorem \ref{gausscond} relies on the following proposition.

\medskip

\begin{proposition}\label{gaussprop} Let $F=(X_1,X_2,\dots,X_d,Z)\in\mathcal{F}_{d,\lambda}$ and $F_t=F+\sqrt{2t}F^*$ (using independent sum) for every $t>0$. Then the function $$\Lambda_F(t)=\mathbb D(T(F_t))-(d/2)\mathbb D(X_1+\sqrt{2t}X_1^*,Z+\sqrt{2t}Z^*)$$ satisfies $\Lambda_F'(0)\geq 0$. If $F$ is not Gaussian then $\Lambda_F'(t)> 0$ for some $t\geq 0$.
\end{proposition}

\medskip

Note that the choice of $F_t$ comes from the heat equation in $\mathbb{R}^d$ (see chapter \ref{heatedrandom}). We first show that Proposition \ref{gaussprop} implies Theorem \ref{gausscond}. The joint distribution $F_t$ does not satisfy the covariance conditions of definition \ref{stardist} but it is clear that the scaled version $F_t(1+2t)^{-1/2}$ is in $\mathcal{F}_{d,\lambda}$. Notice that scaling does not change the differential entropy formula because the extra additive constants coming from scaling exactly cancel each other. By using the claim for $G=F_t(1+2t)^{-1/2}$ we obtain from $\Lambda_G'(0)\geq 0$ that $\Lambda_F'(t)\geq 0$ holds for $F$ with every $t\geq 0$. Since $F_t(1+2t)^{-1/2}=F(1+2t)^{-1/2}+F^*(2t/(1+2t))^{1/2}$ converges to $F^*$ as $t\rightarrow \infty$, we obtain that $F^*$ maximizes $\mathcal{D}$. To see that only the Gaussian system $F^*$ attains the maximum assume that $F$ attains the maximum.  In this case $\Lambda_F'(t)\geq 0$ is only possible if $\Lambda_F'(t)=0$ holds for every $t\geq 0$. This implies that $F$ is Gaussian by the second part of Proposition \ref{gaussprop}. 

It remains to prove Proposition \ref{gaussprop}. We start with some notation and lemmas. Assume that the measure $\mu$ is the distribution and $f$ is the density function of $B=T(F)$.  
 We can choose  a matrix $Q\in \mathbb R^{(d+1)\times d}$ such that $X_i=\sum_k Q_{i,k}B_k$ holds for $i=1, \ldots, d$, and $Z=\sum_k Q_{d+1, k} B_k$ is also satisfied. We define $v_i=(Q_{i, 1}, \ldots, Q_{i, d})$ for $i=1, \ldots, d$ and $w=(Q_{d+1, 1}, \ldots, Q_{d+1, d})$. Notice that the covariance matrix of $F$ is $QQ^T$.
  Hence the vectors  $\{v_i\}_{i=1}^d$ and $w$ are unit vectors such that $(v_i,v_j)=(\lambda^2-d)d^{-1}(d-1)^{-1}$ and $(v_i,w)=\lambda/d$ for every $1\leq i,j\leq d$. The joint distribution $((v_1,B),(v_2,B),\dots,(v_d,B),(w,B))$ is the same as $F$. We have that $\mu$ is invariant with respect to every orthogonal transformation permuting the system $\{v_i\}_{i=1}^d$ and fixing $w$. Furthermore we have that $\mu$ projected to the space spanned by $v_i$ and $w$ is invariant with respect to the reflection interchanging $v_i$ and $w$. We can choose two real numbers $\alpha$ and $\beta$ such that the vectors $a_i=\alpha w+\beta v_i$ and $b_i=\alpha v_i+\beta w$ satisfy $(a_i,b_i)=0$ and $\|a_i\|_2=\|b_i\|_2=1$ for every $1\leq i\leq d$ (in particular, we need that $(\alpha^2+\beta^2)\lambda/d+2\alpha \beta=0$, and $(\alpha^2+\beta^2)+2\alpha\beta\lambda/d=1$). The choice of $\alpha$ and $\beta$ is unique up to multiplying both by $-1$ or switching them. Their values can be determined using elementary geometry.

Note that construction of the vector system $\{v_i\}_{i=1}^d, w, \{a_i\}_{i=1}^d, \{b_i\}_{i=1}^d$ in $\mathbb{R}^d$ is purely linear algebraic. Such system, with the scalar products given above, can be constructed for an arbitrary $|\lambda|\leq d$ however in the case of $|\lambda|\leq2\sqrt{d-1}$ they satisfy a useful geometric property expressed in the following lemma.  

\begin{lemma}\label{wr} If $|\lambda|\leq2\sqrt{d-1}$, then there are numbers $t_1,t_2$ with $t_1,t_2\geq 0$ and $t_1+t_2=1$ such that for every $u\in\mathbb{R}^d$ we have
$$\|u\|_2^2=\sum_{i=1}^d \Bigl(t_1(u,a_i)^2+t_2(u,b_i)^2\Bigr).$$

\end{lemma}

\begin{proof} The proof follows from two observations. The first one is the following. Let $\{v'_i\}_{i=1}^d$ be a system of unit vectors such that all pairwise scalar products are equal and for all $1\leq i\leq d$ we have $(v'_i,w)=c$ for some $c\leq 1$. (The system $\{a_i\}_{i=1}^d$ satisfies the conditions with $c=\alpha+\beta \lambda/d$, and $\{b_i\}_{i=1}^d$ with $c=\alpha\lambda/d+\beta$.) Then
\begin{equation}\label{umbrella}
\sum_{i=1}^d (u,v_i')^2=(1-c^2)d(d-1)^{-1}\|u-(u,w)w\|_2^2+c^2d(u,w)^2
\end{equation}
holds for every $u\in\mathbb{R}^d$. 
To see this, first notice that $\sum_i v_i'=cd w$, which implies that $(v_i', v_j')=(c^2d-1)/(d-1)$ for $1\leq i<j\leq d$. It follows that we can choose $\gamma \in \mathbb R$ such that the equality $(v_i'-\gamma w, v_j'-\gamma w)=0$ holds for $1\leq i<j\leq d$. On the other hand, we have  $\sum_i (u-(u,w)w, v'_i)=0$. Therefore 
\[\sum_{i=1}^d (u, v_i')^2=\sum_{i=1}^d \big((u-(u,w)w, v_i')^2+(u, w)^2(w, v_i')^2\big).\]
The second term is equal to the second term of $\eqref{umbrella}$. In the first term, we can replace $v_i'$ with $v_i'-\gamma w$. If the latter is equal to zero, then we are done. Otherwise, since $\{v_i'-\gamma w\}_{i=1}^d$ is an orthogonal basis in $w^{\bot}$, we obtain 
\[\sum_{i=1}^d (u, v_i')^2=\|v_i'-\gamma w\|_2^2\|u-(u,w)w\|_2^2+c^2d(u,w)^2.\]
On the other hand, for symmetry reasons, $\|v_i'-\gamma w\|_2^2$ does not depend on $i$. By substituting $u=v_1'$ and using $(v_1', v_j')=(c^2d-1)/(d-1)$ again, we get that the value of this constant is the same as in  equation \eqref{umbrella}.

   The second observation says that if $|\lambda|\leq 2\sqrt{d-1}$ then there exist constants $t_1,t_2\geq 0$ with $t_1+t_2=1$ such that $t_1(a_i,w)^2+t_2(b_i,w)^2=1/d$. First of all note the symmetries of the vector system imply that $(a_i,w)^2,(b_i,w)^2$ are independent from $i$. Elementary calculation shows that the two values $(a_i,w)^2$ and $(b_i,w)^2$ are equal to $(1\pm\sqrt{1-(\lambda/d)^2})/2$. This shows the existence of the constants $t_1,t_2$. We get the statement of the lemma by taking the convex combination of (\ref{umbrella}) applied for $\{a_i\}_{i=1}^d$ and $\{b_i\}_{i=1}^d$ with coefficients $t_1$ and $t_2$. 
\end{proof}\hfill $\square$

\medskip

Now we return to the proof of Proposition \ref{gaussprop}. For $1\leq i\leq d$ let $f_i$ denote the orthogonal projection of $f$ to the two dimensional space $V_i=\langle w,v_i\rangle_{\mathbb{R}}$. This means that $$f_i(x)=\int_{z\in V_i^{\perp}}f(x+z)$$ for $x\in V_i$.
Let $T_2:\mathbb{R}^2\rightarrow\mathbb{R}^2$ denote the linear transformation  $T_2(x,y)=(\alpha x+\beta y, \alpha y+\beta x)$ with $\alpha$ and $\beta$ defined above. We have that $f_i$ (when written in the orthonormal basis $a_i,b_i$) is the density function of $T_2(X_i,Z)$. We can write $\Lambda_F(t)$ as 
\begin{equation}
k+\mathbb D(T(F_t))-(d/2)\mathbb D(T_2(X_1+\sqrt{2t}X_1^*,Z+\sqrt{2t}Z^*))
\end{equation}
where the constant $k$ comes from the change of basis $T_2$.
Then by the de Bruijn identity (see equation \eqref{diffheat} and Lemma \ref{korlat})  we get  $$\Lambda_F'(0)=\int_{\mathbb{R}^d} \|\triangledown f\|_2^2/f-(d/2)\int_{V_1} \|\triangledown f_1\|_2^2/f_1.$$ 
From Lemma \ref{wr}  we have that
\begin{equation}
\|\triangledown f\|_2^2=\sum_{i=1}^d (t_1(\partial_{a_i}f)^2+t_2(\partial_{b_i}f)^2)
\end{equation}
holds for some $t_1, t_2\geq 0$ such that $t_1+t_2=1$.
Using
$$\|\triangledown f_1\|_2^2=(\partial_{a_1}f_1)^2+(\partial_{b_1}f_1)^2$$
and the above equations it follows that
$$\Lambda_F'(0)=\sum_{i=1}^d \int_{\mathbb{R}^d} (t_1(\partial_{a_i}f)^2+t_2(\partial_{b_i}f)^2)/f-(d/2)\int_{V_1} ((\partial_{a_1}f_1)^2+(\partial_{b_1}f_1)^2)/f_1.$$ 
By the symmetries of $f$ we have that the terms in the above sum are all the same and thus
\begin{equation}\label{eq:aa11}
\Lambda_F'(0)=d\int_{\mathbb{R}^d} (t_1(\partial_{a_1}f)^2+t_2(\partial_{b_1}f)^2)/f-(d/2)\int_{V_1} ((\partial_{a_1}f_1)^2+(\partial_{b_1}f_1)^2)/f_1.
\end{equation}

\begin{lemma}\label{csch} If $u\in V_1$, then $$\int_{\mathbb{R}^d} (\partial_uf)^2/f \geq \int_{V_1}(\partial_uf_1)^2/f_1.$$ Equality holds if and only if the function $g=\partial_u f(x)/f(x)=\partial_u \log f$ satisfies $g(x)=g(x+z)$ for every pair $x\in\mathbb{R}^d$ and $z\in V_1^\perp$.
\end{lemma}

\begin{proof} $$\int_{\mathbb{R}^d} (\partial_uf)^2/f=\int_{x\in V_1}f_1(x)\int_{z\in V_1^\perp} (f(x+z)/f_1(x))(\partial_u f(x+z)/f(x+z))^2$$
$$\geq \int_{x\in V_1}f_1(x)\Bigl(\int_{z\in V_1^\perp} (f(x+z)/f_1(x))(\partial_u f(x+z)/f(x+z))\Bigr)^2$$
$$=\int_{x\in V_1}f_1(x)\Bigl(\int_{z\in V_1^\perp} (\partial_u f(x+z)/f_1(x))\Bigr)^2=\int_{V_1} (\partial_uf_1)^2/f_1.$$ To see the inequality in the above calculation notice that $f(x+z)/f_1(x)$ is the density function of a probability measure on $x+V_1$ . We can apply the Cauchy--Schwarz inequality using this density function to get the inequality.
It also shows that equality holds in the statement of the lemma if and only if $\partial_u f(x+z)/f(x+z)$ is constant almost everywhere on $x+V_1$ for almost every $x$. Since we work with continuous functions the almost can be omitted. \hfill $\square$
\end{proof}

\medskip

We apply Lemma \ref{csch} for $a_1$ and $b_1$ and (\ref{eq:aa11}) to obtain that 
\begin{equation}\label{eq:ab1}\Lambda_F'(0)\geq (dt_1-d/2)\int_{V_1} (\partial_{a_1}f_1)^2/f_1+(dt_2-d/2)\int_{V_1} (\partial_{b_1}f_1)^2/f_1.\end{equation}
Using the symmetry of $(X_1,Z)$ we obtain that 
\begin{equation}\label{eq:ab2}\int_{V_1} (\partial_{a_1}f_1)^2/f_1=\int_{V_1} (\partial_{b_1}f_1)^2/f_1.\end{equation}
It follows that $\Lambda_F'(0)\geq 0$. The proof of the first part of Proposition \ref{gaussprop} is now complete.

We arrived to the second part of Proposition \ref{gaussprop}. Assume that $F$ satisfies $\Lambda_F'(t)=0$ for every $t\geq 0$. Using the notation from Lemma \ref{wr} we have by $t_1+t_2=1$ that at least one of $t_1>0$ and $t_2>0$ holds. Without loss of generality we assume that $t_1>0$. Let $g_t:\mathbb{R}^d\to\mathbb{R}$ denote the logarithm of the density function of $B_t=T(F_t)$. 
This implies by Lemma \ref{csch} that $\partial_{a_1}g_t$ satisfies the property that $\partial_{a_1}g_t(x)=\partial_{a_1}g_t(x+z)$ holds whenever $x\in \mathbb R^d$ and $z\in V_1^{\perp}$. For convenience, we will write every element $x\in \mathbb R^d$ as a triple $(\alpha(x), \beta(x), \gamma(x))$, where $\alpha(x)=(x, a_1)$, $\beta(x)=(x, b_1)$ and $\gamma(x)$ is the projection of $x$ to $V_1^{\perp}$. Using this notation, we have that $\partial_{a_1}g_t(x)=h_t(\alpha(x), \beta(x))$. This means that there exists a function $\widehat{h_t}: \mathbb R^2\rightarrow \mathbb R$ such that $\partial_{a_1} \widehat{ h_t}(\alpha(x), \beta(x))=\partial_{a_1}g_t(x)$. We have by $\partial_{a_1}(g_t-\widehat{h_t})=0$ that $g_t(x)-\widehat{h_t}(\alpha(x), \beta(x))=s_t(\beta(x), \gamma(x))$ for some function $s_t$. We obtain that the density function of $B_t$ can be written in the following form.
\[\exp(\widehat{h_t}(\alpha(x), \beta(x)))\cdot \exp(s_t(\beta(x), \gamma(x)).\]
In other words, this means that the random variables of $\alpha(B_t)$ and $\gamma(B_t)$ are conditionally independent with respect to $\beta(B_t)$. This implies by Lemma \ref{twoposs} that one of the following two possibilities holds: either $\alpha(B)$ is independent of $(\beta(B), \gamma(B))$ or $\gamma(B)$ is independent of $(\alpha(B), \beta(B))$. In the first case we obtain (using the terminology of Lemma \ref{indepgauss}) that $a_1$ is an independent direction for $B$. By symmetries of $B$ we obtain that $\{a_i\}_{i=1}^d$ are all independent directions for $B$. If $t_1<1$ then $(a_i,a_j)\neq 0$ for every pair $1\leq i<j\leq d$ and Lemma \ref{indepgauss} finishes the proof. If $t_1=1$, then $\{a_i\}_{i=1}^d$ is an orthonormal basis in $\mathbb{R}^d$ and $b_1=\sum_{i=2}^{d-1}a_i(d-1)^{-1/2}$. We have that $(B,a_i)$ are identically distributed independent random variables and that $(B,b_1)=\sum_{i=2}^{d-1}(B,a_i)(d-1)^{-1/2}$ has the same distribution. This is only possible if this $(B,a_i)$ is normal for every $i$.
We obtain that $B$ is Gaussian.

In the case when $\gamma(B)$ is independent of $(\alpha(B), \beta(B))$ we have that $\partial_u g(x)=\partial_ug(x+z)$ holds for every triple $u\in V_1, x\in\mathbb{R}^d,z\in V_1^\perp$. The symmetries of $f$ imply that $\partial_u g(x)=\partial_ug(x+z)$ holds for every triple $u\in V_i, x\in\mathbb{R}^d,z\in V_i^\perp$.
Let $r_i$ denote the orthogonal projection of $a_i$ to $w^\perp$ for $1\leq i\leq d$. Note that the vector system $\{r_i\}_{i=1}^d$ is completely symmetric in the sense that the origin is the center of a regular simplex whose vertices are given by these vectors. 
We have for every $1\leq i\leq d$ that $\partial_{r_i} g(x)=h_i((x,r_i),(x,w))$ for some two variable function $h_i:\mathbb{R}^2\to\mathbb{R}$. The symmetries of $f$ imply that $h_i$ does not depend on $i$ and thus $h_i=h$ for some $h$ for every $i$. 

The next step is to prove that $h(x,y)=xh^*(y)$ for some one variable function $h^*$.
We have by $\sum_{i=1}^d r_i=0$ that $\sum_{i=1}^d\partial_{r_i}g(x)=0$ and thus $\sum_{i=1}^d h_i((x,r_i),(x,w))=0$ holds for every $x\in\mathbb{R}^d$. For arbitrary numbers $x_1,x_2,\dots,x_d,y\in\mathbb{R}$ with $\sum_{i=1}^d x_i=0$ we can choose $x\in\mathbb{R}^d$ such that $(x,r_i)=x_i$ and $(x,w)=y$. It follows that $\sum_{i=1}^d h(x_i,y)=0$ holds for arbitrary numbers with $\sum_{i=1}^d x_i=0$. Assume first that all $x_i$ is $0$ then we have that $dh(0,w)=0$ and thus $h(0,w)=0$ for every $w$. Then assume that $x_1=a,x_2=-a$ and $x_i=0$ if $i\geq 3$. We obtain that $h(a,w)+h(-a,w)=0$ and thus $h(-a,w)=-h(a,w)$ holds for every $a$ and $w$. Finally let $x_1=a,x_2=b,x_3=-a-b$ and $x_i=0$ if $i\geq 4$. We obtain that $h(a,w)+h(b,w)=-h(-a-b,w)=h(a+b,w)$.
Since $h$ is additive and continuous in the first coordinate a well known fact implies that $h$ is a linear function in the first coordinate and thus we obtain $h(x,y)=xh^*(y)$. 

Now we have $\partial_{r_i} g(x)=(x,r_i)h^*((x,w))$ for every $1\leq i\leq d$. It is easy to see that this implies that $g(x)=\|x-(x,w)w\|_2^2h^*((x,w))+c^*((x,w))$ where $c^*:\mathbb{R}\to\mathbb{R}$ is some function.
We have that $\partial_w g(x)=\|x-w(x,w)\|_2^2(h^*)'((x,w))+(c^*)'((x,w))$.
On the other hand we have that $\partial_w g(x)=\partial_w g(x+z)$ holds whenever $z\in V_i^\perp$ for every $1\leq i\leq d$. It follows that $(h^*)'=0$ everywhere and so $g(x)=c\|x-(x,w)w\|_2^2+c^*((x,w))$ with some constant $c$. Now from $f(x)=\exp(c\|x-(x,w)w\|_2^2+c^*((x,w)))$ we have that $B-(B,w)w$ and $(B,w)$ are independent random variables. Furthermore $B-(B,w)w$ is a Gaussian distribution concentrated on the orthogonal space of $w$. This means that the pair $(B,r_1),(B,w)$ of random variables is independent and $(B,r_1)$ is Gaussian. We know that $(B,v_1)$ is a linear combination of $(B,r_1)$ and $(B,w)$ (with a non-zero coefficient for $(B,r_1)$) and its distribution is the same as the distribution of $(B,w)$ (here we use the symmetries of $B$). It follows that $(B,w)$ is also Gaussian and thus $c^*((x,w))=c_2(x,w)^2+c_3$ for some constants $c_2,c_3$. Thus we have that $B$ is a Gaussian joint distribution implying that $F$ is also joint Gaussian, as it is a linear function of $B$. \hfill $\square$

\medskip

\noindent{\bf Proof of Theorem \ref{main} and Theorem \ref{mainlim}}~~From Proposition \ref{lem:equiv} we have that theorem \ref{mainlim} implies theorem \ref{main}. Let $\mu$ be a smooth typical eigenvector process corresponding to eigenvalue $\lambda$ represented by a system of random variables $\{X_v\}_{v\in V_d}$. According to the results in chapter \ref{smooth} it is enough to show that $\mu$ is a Gaussian wave. We have by Lemma \ref{lem:friedman} that $\lambda\in [-2\sqrt{d-1},2\sqrt{d-1}]$. Let $F=\{X_v\}_{v\in C}$. It is clear that $F\in\mathcal{F}_{d,\lambda}$. We have by Theorem \ref{thm:imprmain} that $\mathcal{D}(F)\geq\mathcal{D}(F^*)$ and thus by Theorem \ref{gausscond} we obtain that $F=F^*$. Again by Theorem \ref{thm:imprmain} we have that $\mu$ is $2$-Markov and so $\{X_v\}_{v\in V_d}$ can be obtained by iterating conditionally independent couplings of $C$ along edges (see chapter \ref{chap:impr}). This shows the Gaussianity of the whole system $\{X_v\}_{v\in V}$.

\section{Appendix A: On heated random variables}\label{heatedrandom}

Let $X$ be a random variable with values in $\mathbb{R}^n$ and let $M$ be the standard normal distribution on $\mathbb{R}^n$. Let $f_t$ denote the density function of $X+\sqrt{2t}M$ and let $\mu_t$ denote the corresponding measure on $\mathbb{R}^n$. The standard heat equation says that $\partial_tf_t=\triangle f_t$ holds for every $t>0$. It is useful to compute the variation of the differential entropy $\mathbb D(f_t)$. The de Bruijn identity (see e.g.\ \cite{cover}) says  that 
\begin{equation}\label{diffheat}
\partial_t(\mathbb D(f_t))=\partial_t \int_{\mathbb{R}^n} -f_t\log f_t=\int_{\mathbb{R}^n} -\triangle f_t (1+\log f_t)=\int_{\mathbb{R}^n} \|\triangledown f_t\|_2^2/f_t.
\end{equation}


However the validity of (\ref{diffheat}) relies on the fact that both $\partial_i f_t$ and $\partial_i f_t\log f_t$ vanish at infinity for every $1\leq i\leq n$. This fact is proved in Lemma \ref{korlat}.
Notice that if $X$ has finite variance then also $X+\sqrt{2t}M$ has finite variance and if $t>0$ then $\mathbb D(f_t)$ is a finite quantity (Lemma \ref{lem:veges}).

In general if  $\sigma>0$ then the density function $f$ of $X+\sigma M$ is smooth, non-vanishing  and analytic restricted to every line in $\mathbb{R}^n$. More precisely, if $p,q\in\mathbb{R}^n$ then the real function $\lambda\mapsto f(p+\lambda(q-p))$ extends to an entire analytic function on $\mathbb{C}$. Furthermore every partial derivative of $f$ has this property. In the rest of this appendix we prove several other facts about heated random variables.

\begin{lemma}\label{appineq}  Let $X$ be a random variable with values in $\mathbb{R}^n$ and let $M$ be the standard normal distribution on $\mathbb{R}^n$. Let $f$ be the density function of the independent sum $X+\sigma M$ for some $\sigma>0$. Then for every $1\leq i\leq n$ and $x\in\mathbb{R}^n$ we have that $|\partial_i f(x)|\leq f(x)a(1+|\log(bf(x))|^{1/2})$ for some constants $a,b$ depending on $n$ and $\sigma$.
\end{lemma}

\begin{proof} Let $\Phi$ be the density function of $\sigma M$ and let $\mu$ be the distribution of $X$. Let $r\in\mathbb{R}^+$ the smallest positive real number such that $\partial_i\Phi(z)\leq |f(x)|$ for every $z$ satisfying $|z|\geq r$.
It can be shown that $r\leq c_2(1+|\log (c_1f(x))|^{1/2})$ for some constants (depending on $n$ and $\sigma$).
Let $D=\{z : |x-z|\leq r\}$. We have that $\int_{z\in \overline D} \partial_i\Phi(x-z) d\mu\leq |f(x)|$. On the other hand, by $f(x)=\int_{z\in \mathbb R} \Phi(x-z) d\mu$ we obtain that $$\int_{z\in D}\partial_i\Phi(x-z) d\mu\leq |f(x)|\max_{|y|\leq r}\partial_i\Phi(y)/\Phi(y).$$ Using that $\partial_if(x)=\int_z \partial_i\Phi(x-z)d\mu$ and that $\partial_i\Phi(z)/\Phi(z)=O(z)$ the proof is complete.
\hfill $\square$
\end{proof} 

\begin{lemma}\label{korlat} Let $X$ be a random variable with values in $\mathbb{R}^n$ and let $M$ be a random variable with  standard normal distribution on $\mathbb{R}^n$. Let $f$ be the density function of the independent sum $X+\sigma M$ for some $\sigma>0$. Then for $1\leq i\leq n$ the functions $f$, $\partial_i f$ and $\partial_i f\log f$  vanish at infinity. 
\end{lemma}

\begin{proof} We start with $f$. For contradiction, assume that $D=\{x : f(z)\geq c\}$ is unbounded for some $c>0$. Let $\Phi$ be the density function of $\sigma M$ and let us choose $r\in\mathbb{R}^+$ such that $\Phi(x)\leq c/2$ whenever $\|x\|_2\geq r$. Let $\mu$ be the probability distribution of $X$. We have that if $f(x)\geq c$ and $Q_x=\{z: \|z-x\|_2< r\}$, then $\int_{z\in Q_x}\Phi(z-x) d\mu\geq c/2$ and thus $\mu(Q_x)\geq \|\Phi\|_\infty^{-1}c/2$. From the unboundedness of $D$ we conclude that there is an infinite set of points $\{p_i\}_{i=1}^\infty$ in $D$ such that $\|p_i-p_j\|_2>2r$ holds for every pair $i\neq j$ in $\mathbb{N}$. This contradicts the fact that $\mu$ is finite. The statement for $\partial_i f$ and $\partial_i f\log f$ follows from Lemma \ref{appineq}.\hfill $\square$

\end{proof} 

\begin{lemma}\label{lem:veges} Let $X$ be a random variable with values in $\mathbb R^n$ with finite covariance matrix. Let $M$ be independent of $X$, with standard normal distribution on $\mathbb R^n$. Then, for every  $\sigma>0$,  $X+\sigma M$ has finite differential entropy.
\end{lemma}

\begin{proof} The random variable $X+\sigma M$ has finite covariance matrix. As it is well-known, among the distributions with a given covariance matrix, Gaussian distribution maximizes differential entropy. Hence $\mathbb D(X+\sigma M)<\infty$. On the other hand, as in the previous lemma, let $f$ be the density function of $X+\sigma M$. Lemma \ref{korlat} implies that $\{t: f(t)>1\}$ is a compact set. The continuity of $f$ implies that  $\int_{\mathbb R^n} f(t) \log f (t) dt<\infty$. Thus we also have $\mathbb D(X+\sigma M)>-\infty$.
\end{proof}\hfill $\square$

\begin{lemma}\label{arany} Let $X$ be a random variable with values in $\mathbb{R}^n$  and let $M$ be the standard normal distribution on $\mathbb{R}^n$. Let $f$ be the density function of the independent sum $X+\sigma M$ for some $\sigma>0$. Then for every $\varepsilon>0$ there is $\varepsilon'>0$ such that for every pair $a,b\in\mathbb{R}^n$ with $\|a-b\|_2=r\leq\varepsilon'$ and $f(b)>c$ we have that $f(a)/f(b)\geq 1-\varepsilon$, where $c= \sigma^{-(n-1)/2}\exp(-r^{-1}/(16\sigma^2)).$
\end{lemma}

\begin{proof} We start by general estimates for a pair $a,b\in\mathbb{R}^n$ with $r=\|a-b\|_2\leq 1/4$. We have that $f(x)=\int_{y\in\mathbb{R}^n}\Phi(x-y) d\mu$ where $\mu$ is the probability distribution of $X$ and $\Phi$ is the density function of $\sigma M$. Let $D=\{z : \|z-a\|_2\leq r^{-1/2}\}$. Let $f_1(x)=\int_{y\in\mathbb{R}^n}1_D\Phi(x-y) d\mu$ and $f_2(x)=f(x)-f_1(x)$. We have that $f_2(x)\leq\sup_{z\in\overline{D}}\Phi(x-z)$. It follows that $|f_2(a)|,|f_2(b)|\leq (2\pi \sigma^2)^{-(n-1)/2}\Phi_0(r^{-1/2}-r)$ where $\Phi_0$ is the density function of the one dimensional normal distribution $N(0,\sigma)$. Thus using $r^{-1/2}-r\geq r^{-1/2}/2$ and $1/\sqrt{2\pi}<1$ we have 
\begin{equation}\label{eq:f2}|f_2(a)|,|f_2(b)|\leq(2\pi \sigma^2)^{-(n-1)/2}\Phi_0(r^{-1/2}/2)<c^2.\end{equation}
To estimate $f_1(a)/f_1(b)$ we use 
$$\min_{z\in D}\Phi(z-a)/\Phi(z-b)\leq f_1(a)/f_1(b).$$ 
From $$\Phi(z-a)/\Phi(z-b)=\exp((\|z-b\|_2^2/2-\|z-a\|_2^2/2)/\sigma^2)=\exp(((a-z,b-a)+r^2/2)/\sigma^2)$$
it follows that
\begin{equation}\label{eq:fab}f_1(a)/f_1(b) \geq\exp((-r^{1/2}+r^2/2)/\sigma^2)\geq 1-r^{1/2}/\sigma^2 .\end{equation}
Inequality \eqref{eq:f2}  implies that $f(b)=f_1(b)+f_2(b)\leq f_1(b)+c^2$. Using this and $f_1(a)\leq f(a)$ we obtain  $$f(a)/f(b)\geq f_1(a)/(f_1(b)+c^2)=(f_1(a)/f_1(b))(1+c^2/f_1(b))^{-1}.$$ If $f(b)>c$, then $f_1(b)=f(b)-f_2(b)>c-c^2$ and thus by \eqref{eq:fab} we get  $$f(a)/f(b)>(f_1(a)/f_1(b))(1+c^2/(c-c^2))^{-1}\geq (1-r^{1/2})(1-c).$$
The quantity $c$ goes to $0$ with $r$ and so if $r$ is small enough we have that $f(a)/f(b)\geq 1-\varepsilon$.
\end{proof} \hfill $\square$

\medskip

Let $X$ be a random variable with values in $\mathbb{R}^d$. We say that $v\in\mathbb{R}^d$ is an independent direction for $X$ if the $\mathbb{R}$-valued random variable $(X,v)$ is independent from the projection of $X$ to the $d-1$ dimensional space $v^{\perp}$. Note that every direction is independent for the standard normal distribution. 

\begin{lemma}\label{indepgauss} Let $X$ be a random variable with values in $\mathbb{R}^d$ and with $\mathbb{E}(X)=0$. Assume that $\{v_i\}_{i=1}^d$ is a basis in $\mathbb{R}^d$ such that each $v_i$ is an independent direction for $X$ and furthermore for every $1\leq i\leq d$ there is $1\leq j_i\leq d$ such that $(v_i,v_{j_i})\neq 0$. Then $X$ is Gaussian.
\end{lemma}

\begin{proof} Let $N$ be the standard normal distribution. It is clear that the independent sum $X+\varepsilon N$ has the same independence property as $X$ for every $\varepsilon\geq 0$. Furthermore it is enough to prove that the heated version $X+\varepsilon N$ of $X$ is Gaussian for every $\varepsilon>0$. Let us fix $\varepsilon>0$. The advantage of working with $X+\varepsilon N$ is that it has a strictly positive smooth density function $f$ on $\mathbb{R}^d$ and so we can work with logarithms and with partial derivatives. The independence property now says that $\partial_{v_i} \log f (x)$ is equal to $h_i((x,v_i))$ for some smooth function $h_i:\mathbb{R}\to\mathbb{R}$. 
We obtain that
$$\partial_{v_{j_i}}\partial_{v_i}f=(v_i,v_{j_i})h_i''((x,v_i))~~,~~\partial_{v_i}\partial_{v_{j_i}}f=(v_i,v_{j_i})h_{j_i}''((x,v_{j_i})).
$$
and so $h_i''((x,v_i))=h_{j_i}''((x,v_{j_i}))$. Since $v_i$ and $v_{j_i}$ are independent for every pair $a,b\in\mathbb{R}$ there is $x\in\mathbb{R}^d$ such that $(x,v_i)=a$ and $(x,v_{j_i})=b$. This implies that $h_i''(a)=h_{j_i}''(b)$ holds for every $a,b$ and so each $h_i''$ is a constant function for every $i$.
Consequently $h_i$ is linear for every $i$ and thus (using that $\{v_i\}_{i=1}^d$ is a basis) $\partial_u \log f$ is a linear function for every $u\in\mathbb{R}^d$. It follows that $\log f$ satisfies $\partial_{u_1}\partial_{u_2}\partial_{u_3} \log f =0$ for every $u_1,u_2,u_3\in\mathbb{R}^d$. This means that $\log f$ is given by a quadratic $d$-variate polynomial $Q$ on $\mathbb{R}^d$. We obtain that $f=c\exp(Q(x))$ holds and thus $f$ is the density function of some Gaussian distribution. 

\end{proof}

\begin{lemma}\label{twoposs} Let $(X,Y,Z)$ be a joint distribution with $X,Y\in\mathbb{R}$ and $Z\in\mathbb{R}^{d-2}$. Let $(X_t,Y_t,Z_t)$ be the triple obtained by running the heat equation for time $t$ with $X_0=X,Y_0=Y$ and $Z_0=Z$.  Assume that for every $t\geq 0$ we have that $X_t$ and $Z_t$ are conditionally independent with respect to $Y_t$. Then either $(X,Y)$ is independent from $Z$ or $(Y,Z)$ is independent from $X$.
\end{lemma}

\begin{proof}
We parametrize $\mathbb{R}^d$ with triples $(x,y,\underline{z})$ where the coordinates $x$ and $y$ are real numbers and $\underline z$ is a $d-2$-dimensional vector. By $\Delta_z$ we mean the sum of the second partial derivates with respect to the coordinates belonging to $\underline{z}$. We denote by $f_t,h_t,g_t$ and $m_t$ the density functions of  $(X_t,Y_t,Z_t),(X_t,Y_t),(Z_t,Y_t)$ and $Y_t$, respectively.  We also introduce $s_t(\underline z, y)= g_t(\underline z, y)/m_t(y)$. Using the conditional independence and the heat equation we obtain the following equations.
\[f_t(x, y, \underline z)=h_t(x, y)g_t(\underline z,y)/m_t(y)=h_t(x,y)s_t(\underline z, y).\]
\[ \partial_tf_t=\Delta f_t; \qquad \partial_th_t=\Delta h_t; \qquad  \partial_tg_t=\Delta g_t=\partial_{yy}g_t+\Delta_z g_t; \qquad \partial_tm_t=\Delta m_t. \]

By abusing the notation we will omit  $t$ from $f_t,h_t,g_t,m_t$ and $s_t$ in the following calculations.
We start from the first equality, and use the other three one after the other. 
\[\partial_{xx} f+\partial_{yy}f+\Delta_{z}f=\partial_th \cdot s+h\cdot\partial_ts.\]
\[\partial_{xx}h\cdot s+\partial_{yy}f+\Delta_{z}f=\partial_{xx}h\cdot s+\partial_{yy}h\cdot s+h\cdot \partial_ts.\]
\[\partial_{yy}f+h\cdot \Delta_{z}g/m=\partial_{yy}h\cdot s+h\cdot\big((\partial_tg)/m-(g\cdot \partial_tm)/m^2\big). \]
\begin{equation}\label{eq:partial1}\partial_{yy}h\cdot s+2\partial_{y}h\cdot \partial_{y} s+h\cdot\partial_{yy} s= \partial_{yy}h\cdot s+(h\cdot \partial_{yy}g)/m-(gh\cdot\partial_{yy}m)/m^2.\end{equation}
Before continuing this, we calculate the partial derivatives of $s$ with respect to $y$.  
\[\partial_{y}s=(\partial_yg)/m-(g\cdot\partial_ym)/m^2.\]
\[\partial_{yy}s=(\partial_{yy}g)/m-2(\partial_y g\cdot \partial_y m)/m^2-(g\cdot\partial_{yy}m)/m^2+2g\cdot(\partial_y m)^2/m^3.\]
Now we substitute this into equation \eqref{eq:partial1}.
\begin{multline*}2(\partial_yh\cdot\partial_yg)/m-2g(\partial_yh\cdot\partial_ ym)/m^2+(h\cdot \partial_{yy}g)/m\\-2h(\partial_y g\cdot \partial_y m)/m^2-(gh\cdot \partial_{yy}m)/m^2+2hg(\partial_ym)^2/m^3=(h\cdot \partial_{yy}g)/m-(gh\cdot\partial_{yy}m)/m^2.\end{multline*}
\[(\partial_yh\cdot\partial_yg)/m-(g\cdot\partial_y h+h\cdot\partial_y g)\cdot(\partial_ym)/m^2+gh(\partial_ ym)^2/m^3=0.\]
\[(\partial_yh\cdot\partial_yg)/m^2-(g\cdot\partial_y h+h\cdot\partial_y g)\cdot(\partial_ym)/m^3+gh(\partial_ ym)^2/m^4=0.\]
\[\Big((\partial_y h)/m-(h\cdot\partial_y m)/m^2\Big)\Big((\partial_yg)/m-(g\cdot \partial_y m)/m^2\Big)=0.\]
\[\big(\partial_y(h/m)\big)\cdot \big(\partial_y(g/m)\big)=0.\] 
We obtain that at least one of $\partial_y(h/m)=0$ and $\partial_y(g/m)=0$ holds on an open set. Assume that (without loss of generality) $\partial_y(h/m)=0$ holds on an open set $U$ of the domain $D$ of $h/m$ which is $\mathbb{R}^2$. This is equivalent to $m\partial_y h - h\partial_ym=0$ on $U$. Let $p\in U$ and $q\in D$ be arbitrary. Let us define the function $r:\mathbb{R}\to\mathbb{R}$ by $r(\lambda)=(m\partial_yh-h\partial_y m)(p+\lambda(q-p))$. Then $r$ has an analytic extension to $\mathbb{C}$. In addition, $r=0$ in a small neighborhood of $0$ in $\mathbb{R}$. It follows that $r$ is constant $0$ and thus $\partial_y (h/m)$ is $0$ at every $q\in D$. This implies that $X_t$ is independent of $Y_t$. Similarly if $\partial_y(g/m)=0$ holds on an open set we obtain that $Z_t$ is independent of $Y_t$. The conditional independence of $X_t$ and $Z_t$ with respect to $Y_t$ finishes the proof. 
\hfill $\square$

\end{proof}

\section{Appendix B: Differential entropy}

Differential entropy is defined as follows for absolutely continuous random vectors. Some properties of the discrete entropy are preserved (e.g.\ it is additive if we put together independent random variables), others do not hold any more; an essential difference is that differential entropy does not have to be nonnegative.

\begin{definition}Let $(X_1, X_2, \ldots, X_n)$ be a family of random variables. Suppose that their joint distribution is absolutely continuous, and their joint density function is $f$. Then their differential entropy is defined as follows (provided that the integral exists): 
\[\mathbb D(X_1, X_2, \ldots, X_n)=-\int_{\mathbb R^n} f(t_1, \ldots, t_n) \log f(t_1, \ldots, t_n)dt_1\ldots dt_n.\]
\end{definition}

To see the connection between entropy in the discrete case and differential entropy, recall Theorem 9.3.1\ from \cite{cover}. This says that if we divide the range of $X$ into bins of length $\delta$, and $X^{\delta}$ denotes the quantized version of $X$ with respect to this grid, then 
\[\mathbb H(X^{\delta})+\log \delta\rightarrow \mathbb H(X) \]
as $\delta\rightarrow 0$, assuming that the density of $X$ is Riemann integrable. 

The following well-known lemma shows how the differential entropy is modified when we apply a linear transformation to the random vector (see e.g. corollary to Theorem 9.6.4 in \cite{cover}. 

\begin{lemma}\label{difftrdet} Let $\underline X=(X_1, \ldots, X_n)$ be a family of random variables, and let $A\in \mathbb R^{n\times n}$ be an invertible matrix. Then 
\[\mathbb D(A\underline X)=\mathbb D(\underline X)+\log |\!\det (A)|.\]

\end{lemma}

The following lemma is equivalent to the fact that the nonnegativity of conditional mutual information holds for differential entropy as well. We include a proof for completeness. 
\begin{lemma}\label{diffincform} Let $X, Y, Z$ be random variables such that their differential entropy exist. Then we have
\[\mathbb D(X, Y, Z)\leq \mathbb D(X, Z)+\mathbb D(Y, Z)-\mathbb D(Z).\]

\end{lemma}

\begin{proof}
Let $g(x,y)=f(x, y, z)/f(z)$ on the support of $Z$, and 0 otherwise. Then $g$ is a density function on $\mathbb R^2$. As the nonnegativity of mutual information is satisfied for differential entropy (see e.g.\ Corollary to Theorem 9.6.1.\ in \cite{cover}), we have
\[-\int g(x,y) \log g(x,y) dx dy \leq -\int g_1(x)\log g_1(x) dx-\int g_2(y) \log g_2(y) dy, \]
where $g_1$ and $g_2$ are the marginal densities of $g$. Multiplying both sides by $f(z)$ and integrating with respect to $z$ we get the statement of the lemma.
\end{proof} \hfill $\square$

\section{Appendix C: Factor of i.i.d.\ processes}

Let $f:[0,1]^{V_d}\rightarrow Y$ be a measurable function such that it is invariant under root preserving automorphisms. We can use $f$ to construct an invariant process in the following way. First we put independent uniformly random elements from $[0,1]$ on the vertices of $T_d$. Then, at each vertex $v$, we evaluate $f$ for this random labeling such that the root is placed to $v$.  If $f$ depends only on finitely many coordinates, then the corresponding process is called a block factor of i.i.d.\ process. 

\begin{proposition}\label{prop:elso}
If $\{X_v\}_{v\in V_d}$ is a real-valued typical process and $\{Y_v\}_{v\in V_d}$ is a weak limit of factor of i.i.d.\ processes, then their independent sum $\{X_v+Y_v\}_{v\in V_d}$ is a typical process. 
\end{proposition}
\begin{proof}
Notice that the family of typical processes is closed with respect to the weak topology. On the other hand, every process that is a weak limit of factor i.i.d.\ processes is also a weak limit of block factor of i.i.d.\ processes \cite{russ}. Hence it is enough to prove the statement in the case when $\{Y_v\}_{v\in V_d}$ is a block factor of i.i.d.\ process. It is well known that block factor of i.i.d.\ processes can be approximated with the corresponding local algorithm computed on graphs with sufficiently large (essential) girth. Since large random $d$-regular graphs have large essential girth, the independent sum of the approximation of $\{X_v\}_{v\in V_d}$ and the local algorithm approximating  $\{Y_v\}_{v\in V_d}$ is an approximation of the process $\{X_v+Y_v\}_{v\in V_d}$.
\hfill $\square$
\end{proof}

The next proposition was proved by Harangi and Vir\'ag in \cite{harangi}. 
\begin{proposition}\label{prop:hv}
For $|\lambda|\leq 2\sqrt{d-1}$  the unique Gaussian wave $\Psi_\lambda$ is a weak limit of factor of i.i.d processes (but not a factor of i.i.d.\ process itself). 
\end{proposition}

\subsection*{Acknowledgement.} The research leading to these results has received funding from the European Research Council under the European Union's Seventh Framework Programme (FP7/2007-2013) / ERC grant agreement n$^{\circ}$617747. The research was partially supported by the MTA R\'enyi Institute Lend\"ulet Limits of Structures Research Group.

\textsc{\'Agnes Backhausz.} Eotvos Lor\'and University, Budapest, Hungary and MTA Alfr\'ed R\'enyi Institute of Mathematics. \texttt{agnes@math.elte.hu}

\textsc{Bal\'azs Szegedy.} MTA Alfr\'ed R\'enyi Institute of Mathematics, Budapest, Hungary. \\ \texttt{szegedy.balazs@renyi.mta.hu}

\end{document}